\pdfoutput=1
\documentclass[letterpaper, 10 pt, conference]{ieeeconf}%
\usepackage{graphicx}
\usepackage{epsfig}
\usepackage{amsmath}
\usepackage{amssymb}
\usepackage{bbm}
\usepackage{float}
\usepackage{tikz}
\usepackage{algorithm}
\usepackage{algpseudocode}
\usepackage{pgfplots}
\usepackage{gnuplottex}
\usepackage{mathrsfs}
\usepackage{subfigure}
\usepackage{booktabs}
\usepackage{multirow}
\usepackage{amsfonts}
\usepackage{wrapfig}
\usepackage{epstopdf}
\usepackage[T1]{fontenc}

\providecommand{\U}[1]{\protect\rule{.1in}{.1in}}
\providecommand{\U}[1]{\protect\rule{.1in}{.1in}}
\IEEEoverridecommandlockouts

\renewcommand{\baselinestretch}{0.975}
\makeatletter
\let\@ORGmakecaption\@makecaption
\long
\def\@makecaption#1#2{\@ORGmakecaption{#1}{#2}\vskip\belowcaptionskip\relax}
\makeatother
\def\@normalsize{\@setsize\normalsize{10pt}\xpt\@xpt
\abovedisplayskip 10pt plus2pt minus5pt\belowdisplayskip
\abovedisplayskip \abovedisplayshortskip \z@
plus3pt\belowdisplayshortskip 6pt plus3pt
minus3pt\let\@listi\@listI}
\makeatletter
\newcommand\npar{\@startsection{section}{1pt}}
\makeatother

\newcounter{lemcount}
\newenvironment{lemma}[1][]{\refstepcounter{lemcount}\par  {\it{Lemma~\thelemcount.} #1 }
\rmfamily  }{\par}

\newcommand{\qed}{\nobreak \ifvmode \relax \else
\ifdim\lastskip<1.5em \hskip-\lastskip
\hskip1.5em plus0em minus0.5em \fi \nobreak
\vrule height0.4em width0.5em depth0.25em\fi}

\graphicspath{{C:/Users/JoaoFelipe/Desktop/Julia/CDC2016}}

\title{{\LARGE \textbf{Personalized Cancer Therapy Design:
Robustness vs. Optimality}}}
\author{ \parbox{3.5 in}{\centering Julia L. Fleck and Christos G. Cassandras\\
         Division of Systems Engineering\\ and Center for Information and Systems Engineering\\
         Boston University, MA 02446\\
         {\tt\small jfleck@bu.edu,cgc@bu.edu}}
         \thanks{The authors' work is supported in part by NSF under grants CNS-1239021, ECCS-1509084, and IIP-1430145, by AFOSR under grant FA9550-15-1-0471, and by ONR under grant N00014-09-1-1051.}
         \vspace{-4mm}
}

\allowdisplaybreaks
\begin{document}
\maketitle
\thispagestyle{empty}
\pagestyle{empty}
\begin{abstract}
Intermittent Androgen Suppression (IAS) is a treatment strategy for delaying or even preventing time to relapse of advanced prostate cancer. IAS consists of alternating cycles of therapy (in the form of androgen suppression) and off-treatment periods. The level of prostate specific antigen (PSA) in a patient's serum is frequently monitored to determine when the patient will be taken off therapy and when therapy will resume. In spite of extensive recent clinical experience with IAS, the design of an ideal protocol for any given patient remains one of the main challenges associated with effectively implementing this therapy. We use a threshold-based policy for optimal IAS therapy design that is parameterized by lower and upper PSA threshold values and is associated with a cost metric that combines clinically relevant measures of therapy success. We apply Infinitesimal Perturbation Analysis (IPA) to a Stochastic Hybrid Automaton (SHA) model of prostate cancer evolution under IAS and derive unbiased estimators of the cost metric gradient with respect to various model and therapy parameters. These estimators are subsequently used for system analysis. By evaluating sensitivity estimates with respect to several model parameters, we identify critical parameters and demonstrate that relaxing the optimality condition in favor of increased robustness to modeling errors provides an alternative objective to therapy design for at least some patients.
\end{abstract}

\section{Introduction}
Several recent attempts have been made to develop mathematical models that explain the
progression of cancer in patients undergoing therapy so as to improve (and possibly optimize) the effectiveness of such therapy. As an example, prostate cancer is known to be a multistep process, and patients who evolve into a state of metastatic disease are usually submitted to hormone therapy in the form of continuous androgen suppression (CAS) \cite{Harrison2012}. The initial response to CAS is frequently positive, leading to a significant decrease in tumor size; unfortunately, most patients eventually develop resistance and relapse.

Intermittent Androgen Suppression (IAS) is an alternative treatment strategy for delaying or even preventing time to relapse of advanced prostate cancer patients. IAS consists of alternating cycles of therapy (in the form of androgen suppression) and off-treatment periods. The level of prostate specific antigen (PSA) in a patient's serum is frequently monitored to determine when the patient will be taken off therapy and when therapy will resume. In spite of extensive recent clinical experience with IAS, the design of an ideal protocol for any given patient remains one of the main challenges associated with effectively implementing this therapy \cite{Hirata2010}.

Various works have aimed at addressing this challenge, and we briefly review some of them. In \cite{Jackson2004} a model is proposed in which prostate tumors
are composed of two subpopulations of cancer cells, one that is sensitive to
androgen suppression and another that is not, without directly addressing
the issue of IAS therapy design. The authors in \cite{Ideta2008} modeled the evolution of a prostate tumor under IAS using a hybrid dynamical system approach and applied
numerical bifurcation analysis to study the effect of different
therapy protocols on tumor growth and time to relapse. In \cite{Shimada2008} a nonlinear model is developed to explain the competition between different cancer cell
subpopulations, while in \cite{Tao2010} a model based on switched ordinary differential equations is proposed. The authors in \cite{Suzuki2010} developed a piecewise affine system model and formulated the problem of personalized prostate cancer treatment as an optimal control problem. Patient classification is performed in \cite{Hirata2010} using a feedback control system to model the prostate tumor under IAS, and in \cite{Hirata2010B} this work is extended by deriving conditions for patient relapse.

Most of the existing models provide insights into the dynamics of prostate cancer evolution under androgen deprivation, but fail to address the issue of therapy design. Furthermore, previous works that suggest optimal treatment
schemes by classifying patients into groups have been based on more
manageable, albeit less accurate, approaches to nonlinear hybrid dynamical
systems. Addressing this limitation, the authors in \cite{Liu2015} recently proposed a nonlinear hybrid automaton model and performed $\delta $-reachability analysis to identify patient-specific treatment schemes. However, this model did not account for noise and fluctuations inherently associated with cell population dynamics and monitoring of clinical data. In contrast, in \cite{Tanaka2010} a hybrid model of tumor growth under IAS therapy is developed that incorporated stochastic effects, but is not used for personalized therapy design.

A first attempt to define optimal personalized IAS therapy schemes by applying Infinitesimal Perturbation Analysis (IPA) to stochastic models of prostate cancer evolution was reported in \cite{FleckADHS2015}. An IPA-driven gradient-based optimization algorithm was subsequently implemented in \cite{FleckNAHS2016} to adaptively adjust controllable therapy settings so as to improve IAS therapy outcomes. The advantages of these IPA-based approaches stem from the fact that IPA efficiently yields sensitivities with respect to controllable parameters in a therapy (i.e., control policy), which is arguably the ultimate goal of personalized therapy design. More generally, however, IPA yields sensitivity estimates with respect to various model
parameters from actual data, thus allowing critical parameters to be
differentiated from others that are not.

In this paper we build upon the IPA-based methodology from \cite{FleckADHS2015} and \cite{FleckNAHS2016} and focus on the importance of accurate modeling in conjunction with optimal therapy design. In particular, by evaluating sensitivity estimates with respect to several model parameters, we identify critical parameters and verify the extent to which the model from \cite{FleckADHS2015} is robust to them. From a practical perspective, the goal of this paper is to use IPA to explore the tradeoff between system optimality and robustness (or, equivalently, fragility), thus providing valuable insights on modeling and control of cancer progression. Assuming that an underlying, and most likely poorly understood, equilibrium of cancer cell subpopulation dynamics exists at suboptimal therapy settings, we verify that relaxing the optimality condition in favor of increased robustness to modeling errors provides an alternative objective to therapy design for at least some patients.

In Section \ref{Formulation} we present a Stochastic Hybrid Automaton (SHA) model of prostate cancer evolution, along with a threshold-based policy for optimal IAS therapy design.  Section \ref{IPA} reviews a general framework of IPA based on which we derive unbiased IPA estimators for system analysis. In Section \ref{Results} we evaluate sensitivity estimates with respect to several model parameters, identifying critical parameters and verifying the extent to which our SHA model is robust to them. We include final remarks in Section \ref{Conclusions}.

\section{Problem Formulation}
\label{Formulation}

\subsection{Stochastic Model of Prostate Cancer Evolution}

We consider a system composed of a prostate tumor under IAS therapy, which
is modeled as a Stochastic Hybrid Automaton (SHA). Details of the problem
formulation are given in \cite{FleckADHS2015}, but a condensed description of the SHA modeling framework is included here so as to make this paper as self-contained as possible. By adopting a standard SHA definition \cite{Cassandras2008}, a SHA model of
prostate cancer evolution is defined in terms of the following:

A \textbf{discrete state set} $Q=\left\{ q^{ON},q^{OFF}\right\} $, where $%
q^{ON} $ ($q^{OFF}$, respectively) is the on-treatment (off-treatment,
respectively) operational mode of the system. IAS therapy is temporarily suspended when the size of the prostate tumor decreases by a predetermined desirable amount. The reduction in the size of the tumor is estimated in terms of the patient's prostate specific antigen (PSA) level, a biomarker commonly used for monitoring the outcome of hormone therapy. In this context, therapy is suspended when a patient's PSA level reaches a lower threshold value, and reinstated once the size of cancer cell populations has increased considerably, i.e., once the patient's PSA level reaches an upper threshold value.

A \textbf{state space} $X=\left\{ x_{1}\left( t\right) ,x_{2}\left( t\right)
,x_{3}\left( t\right) ,z_{1}\left( t\right) ,z_{2}\left( t\right) \right\} $, defined in terms of the biomarkers commonly monitored during IAS therapy,
as well as \textquotedblleft clock\textquotedblright\ state variables that
measure the time spent by the system in each discrete state. We assume that
prostate tumors are composed of two coexisting subpopulations of cancer
cells, Hormone Sensitive Cells (HSCs) and Castration Resistant Cells (CRCs),
and thus define a state vector $x\left( t\right) =\left[ x_{1}\left(
t\right) ,x_{2}\left( t\right) ,x_{3}\left( t\right) \right] $ with $
x_{i}\left( t\right) \in \mathbb{R}^{+}$, such that $x_{1}\left( t\right) $ is the total population of HSCs, $x_{2}\left( t\right) $ is the total population of CRCs, and $x_{3}\left(t\right) $ is the concentration of androgen in the serum. Prostate cancer cells secrete high levels of PSA, hence a common assumption is that the
serum PSA concentration can be modeled as a linear combination of the cancer
cell subpopulations. It is also frequently assumed that both HSCs and CRCs
secrete PSA equivalently \cite{Ideta2008}, and in this work we adopt these
assumptions. Finally, we define variable $z_{i}\left( t\right) \in \mathbb{R}^{+}$, $i=1,2$, where $z_{1}\left( t\right) $ ($z_{2}\left( t\right) $,
respectively) is the \textquotedblleft clock\textquotedblright\ state
variable corresponding to the time when the system is in state $q^{ON}$ ($
q^{OFF}$, respectively), and is reset to zero every time a state transition
occurs. Setting $z\left( t\right) =\left[ z_{1}\left( t\right) ,z_{2}\left(
t\right) \right] $, the complete state vector is $\left[ x\left( t\right)
,z\left( t\right) \right] $.

An \textbf{admissible control set} $U=\left\{ 0,1\right\} $, such that the
control is defined, at any time $t$, as:

{\small\vspace{-3mm}
\begin{equation}
u\left( x\left( t\right) ,z\left( t\right) \right) \equiv \left\{
\begin{array}{ll}
0 & \text{if }x_{1}\left( t\right) +x_{2}\left( t\right) <\theta _{2}\text{,
}q\left( t\right) =q^{OFF} \\
1 & \text{if }x_{1}\left( t\right) +x_{2}\left( t\right) >\theta _{1}\text{,
}q\left( t\right) =q^{ON}%
\end{array}
\right.  \label{eq: control set}
\end{equation}}
\vspace{-3mm}

This is a simple form of hysteresis control to ensure that androgen
deprivation will be suspended whenever a patient's PSA level drops below a
minimum threshold value, and that treatment will resume once the patient's
PSA\ level reaches a maximum threshold value. To this end, IAS therapy is
viewed as a controlled process characterized by two parameters: $\tilde{\theta} =
\left[ \tilde{\theta} _{1},\tilde{\theta} _{2}\right] \in \Theta $, where $\tilde{\theta} _{1}\in
\left[ \tilde{\theta} _{1}^{\min },\tilde{\theta} _{1}^{\max }\right] $ is the lower
threshold value of the patient's PSA level, and $\tilde{\theta} _{2}\in \left[
\tilde{\theta} _{2}^{\min },\tilde{\theta} _{2}^{\max }\right] $ is the upper threshold
value of the patient's PSA level, with $\tilde{\theta} _{1}^{\max }<\tilde{\theta}
_{2}^{\min }$. An illustrative representation of such threshold-based IAS
therapy scheme is depicted in Fig. \ref{fig: IAS}. Simulation driven by
clinical data \cite{Bruchovsky2006},\cite{Bruchovsky2007} was performed to
generate the plot in Fig. \ref{fig: IAS}, which shows a typical profile of
PSA\ level variations along several treatment cycles.

\begin{figure}
[tbh]
\begin{center}
\includegraphics[
natheight=2.9689in,
natwidth=5.8954in,
height=1.5in,
width=2.9in
]%
{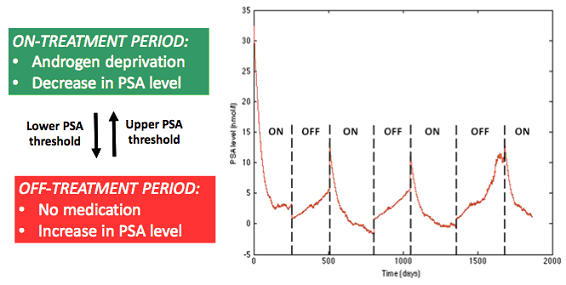}%
\caption{Schematic representation of Intermittent Androgen Suppression (IAS) therapy}%
\label{fig: IAS}%
\end{center}
\end{figure}

An \textbf{event set} $E=\left\{ e_{1},e_{2}\right\} $, where $e_{1}$
corresponds to the condition $\left[ x_{1}\left( t\right) +x_{2}\left(
t\right) =\theta _{1}\text{ from above}\right] $ (i.e., $x_{1}\left(
t^{-}\right) +x_{2}\left( t^{-}\right) >\theta _{1}$) and $e_{2}$
corresponds to the condition $\left[ x_{1}\left( t\right) +x_{2}\left(
t\right) =\theta _{2}\text{ from below}\right] $ (i.e., $x_{1}\left(
t^{-}\right) +x_{2}\left( t^{-}\right) <\theta _{2}$), where the notation $%
t^{-}$ indicates the time instant immediately preceding time $t$.

\textbf{System dynamics} describing the evolution of continuous state variables
over time, as well as the rules for discrete state transitions. The \emph{
continuous (time-driven) dynamics} capture the prostate cancer cell
population dynamics, which are defined in terms of their proliferation,
apoptosis, and conversion rates. As in \cite{FleckADHS2015}, we incorporate
stochastic effects into the deterministic model from \cite{Liu2015} as
follows:

{\small \vspace{-3mm}
\begin{equation}
\begin{array}
[c]{ll}
\dot{x}_{1}(t) & =\alpha_{1}\left[  1+e^{-\left(  x_{3}(t)-k_{1}\right)
k_{2}}\right]  ^{-1}\cdot x_{1}(t)\\
& -\beta_{1}\left[  1+e^{-\left(  x_{3}(t)-k_{3}\right)  k_{4}}\right]
^{-1}\cdot x_{1}(t)\\
& -\left[  m_{1}\left(  1-\frac{x_{3}(t)}{x_{3,0}}\right)  +\lambda
_{1}\right]  \cdot x_{1}(t)\\
& +\mu_{1}+\zeta_{1}(t)
\end{array}
\label{eq: HSC dynamics}
\end{equation}
}

{\small \vspace{-3mm}
\begin{equation}
\begin{array}{ll}
\dot{x}_{2}(t)= & \left[ \alpha _{2}\left( 1-d\frac{x_{3}(t)}{x_{3,0}}%
\right) -\beta _{2}\right] x_{2}(t) \\
& +m_{1}\left( 1-\frac{x_{3}(t)}{x_{3,0}}\right) x_{1}(t)+\zeta _{2}(t)%
\end{array}
\label{eq: CRC dynamics}
\end{equation}}

{\small \vspace{-3mm}
\begin{equation}
\dot{x}_{3}(t)=\left\{
\begin{array}{ll}
-\frac{x_{3}(t)}{\sigma }+\mu _{3}+\zeta _{3}(t) &
\begin{array}{l}
\text{if }x_{1}(t)+x_{2}(t)>\theta _{1} \\
\text{and }q(t)=q^{ON}%
\end{array}%
\text{ } \\
\frac{x_{3,0}-x_{3}(t)}{\sigma }+\mu _{3}+\zeta _{3}(t) &
\begin{array}{l}
\text{if }x_{1}(t)+x_{2}(t)<\theta _{2}\text{ } \\
\text{and }q(t)=q^{OFF}%
\end{array}%
\end{array}%
\right.  \label{eq: x3 dynamics}
\end{equation}}

{\small
\begin{align}
\dot{z}_{1}(t)& =\left\{
\begin{array}{ll}
1 & \text{if }q(t)=q^{ON} \\
0 & \text{otherwise}%
\end{array}%
\right.  \label{eq: z1 dynamics} \\
z_{1}(t^{+})& =%
\begin{array}[t]{ll}
0 &
\begin{array}{l}
\text{if }x_{1}(t)+x_{2}(t)=\theta _{1} \\
\text{and }q(t)=q^{ON}%
\end{array}%
\text{ }%
\end{array}
\notag
\end{align}%
\begin{align}
\dot{z}_{2}(t)& =\left\{
\begin{array}{ll}
1 & \text{if }q(t)=q^{OFF} \\
0 & \text{otherwise}%
\end{array}%
\right.  \label{eq: z2 dynamics} \\
z_{2}(t^{+})& =%
\begin{array}[t]{ll}
0 &
\begin{array}{l}
\text{if }x_{1}(t)+x_{2}(t)=\theta _{2} \\
\text{and }q(t)=q^{OFF}\text{ }%
\end{array}%
\end{array}
\notag
\end{align}} where $\alpha _{1}$ and $\alpha _{2}$ are the HSC proliferation constant and CRC proliferation constant, respectively; $\beta _{1}$ and $\beta _{2}$ are
the HSC apoptosis constant and CRC apoptosis constant, respectively; $k_{1}$
through $k_{4}$ are HSC proliferation and apoptosis exponential constants; $
m_{1}$ is the HSC to CRC conversion constant; $x_{3,0}$ corresponds to the
patient-specific androgen constant; $\sigma $ is the androgen degradation
constant; $\lambda _{1}$ is the HSC basal degradation rate; $\mu _{1}$ and $
\mu _{3}$ are the HSC basal production rate and androgen basal production
rate, respectively. Finally, $\{\zeta _{i}(t)\}$, $i=1,2,3$, are stochastic
processes which we allow to have arbitrary characteristics and only assume
them to be piecewise continuous w.p. 1. The processes $\{\zeta _{i}(t)\}$, $
i=1,2$, represent noise and fluctuations inherently associated with cell
population dynamics, while $\{\zeta _{3}(t)\}$ reflects randomness
associated with monitoring clinical data, more specifically, with monitoring
the patient's androgen level.

It is clear from (\ref{eq: HSC dynamics})-(\ref{eq: x3 dynamics}) that $
x_{1}\left( t\right) $ and $x_{2}\left( t\right) $ are dependent on $
x_{3}\left( t\right) $, whose dynamics are affected by mode transitions. To
make explicit the dependence of $x_{1}\left( t\right) $ and $x_{2}\left(
t\right) $ on the discrete state (mode) $q\left( t\right) $, we let $\tau
_{k}\left( \tilde{\theta} \right) $ be the time of occurrence of the $k$th event (of
any type), and denote the state dynamics over any interevent interval $\left[
\tau _{k}\left( \tilde{\theta} \right) ,\tau _{k+1}\left( \tilde{\theta} \right) \right) $ as
\begin{equation*}
\dot{x}_{n}(t)=f_{k}^{x_{n}}(t)\text{, }\dot{z}_{i}(t)=f_{k}^{z_{i}}(t)\text{
, }n=1,\ldots ,3\text{, }i=1,2
\end{equation*}
We include $\tilde{\theta} $ as an argument to stress the dependence of the event
times on the controllable parameters, but we will subsequently drop this for
ease of notation as long as no confusion arises.

We thus start by assuming $q(t)=q^{ON}$ for $t\in $ $\left[ \tau _{k},\tau
_{k+1}\right) $. Solving (\ref{eq: x3 dynamics}) yields, for $t\in $ $\left[
\tau _{k},\tau _{k+1}\right) $,

{\small \vspace{-3mm}
\begin{equation*}
\begin{array}{ll}
x_{3}(t) & =x_{3}(\tau _{k}^{+})e^{-\left( t-\tau _{k}\right) /\sigma } \\
& +e^{-t/\sigma }\cdot \int_{\tau _{k}}^{t}e^{\varepsilon /\sigma }\left[
\mu _{3}+\zeta _{3}(\varepsilon )\right] d\varepsilon
\end{array}
\end{equation*}}
It is then possible to define, for $t\in $ $\left[ \tau _{k},\tau
_{k+1}\right) $,

{\small \vspace{-3mm}
\begin{equation}
\begin{array}{ll}
h^{ON}\left( t,\tilde{\zeta}_{3}(t)\right) & \equiv x_{3}(\tau
_{k}^{+})e^{-\left( t-\tau _{k}\right) /\sigma } \\
& +\mu _{3}\sigma \lbrack 1-e^{-\left( t-\tau _{k}\right) /\sigma }]+\tilde{
\zeta}_{3}(t)
\end{array}
\label{eq: h_ON}
\end{equation}}
where, for notational simplicity, we let

{\small \vspace{-3mm}
\begin{equation}
\tilde{\zeta}_{3}(t)=\int_{\tau _{k}}^{t}e^{-\left( t-\varepsilon \right)
/\sigma }\zeta _{3}(\varepsilon )d\varepsilon  \label{eq: zeta3_tilda}
\end{equation}}
Next, let $q(t)=q^{OFF}$ for $t\in $ $\left[ \tau _{k},\tau _{k+1}\right) $,
so that (\ref{eq: x3 dynamics}) implies that, for $t\in $ $\left[ \tau
_{k},\tau _{k+1}\right) $,

{\small \vspace{-3mm}
\begin{equation*}
\begin{array}{ll}
x_{3}(t) & =x_{3}(\tau _{k}^{+})e^{-\left( t-\tau _{k}\right) /\sigma } \\
& +(\mu _{3}\sigma +x_{3,0})[1-e^{-\left( t-\tau _{k}\right) /\sigma }]+
\tilde{\zeta}_{3}(t)
\end{array}
\end{equation*}}
Similarly as above, we define, for $t\in $ $\left[ \tau _{k},\tau
_{k+1}\right) $,

{\small \vspace{-3mm}
\begin{equation}
\begin{array}{ll}
h^{OFF}\left( t,\tilde{\zeta}_{3}(t)\right) & \equiv x_{3}(\tau
_{k}^{+})e^{-\left( t-\tau _{k}\right) /\sigma } \\
& +(\mu _{3}\sigma +x_{3,0})[1-e^{-\left( t-\tau _{k}\right) /\sigma }]+
\tilde{\zeta}_{3}(t)
\end{array}
\label{eq: h_OFF}
\end{equation}}
It is then possible to rewrite (\ref{eq: x3 dynamics}) as follows:

{\small \vspace{-3mm}
\begin{equation*}
x_{3}(t)=\left\{
\begin{array}{ll}
h^{ON}\left( t,\tilde{\zeta}_{3}(t)\right) & \text{if }q(t)=q^{ON} \\
h^{OFF}\left( t,\tilde{\zeta}_{3}(t)\right) & \text{if }q(t)=q^{OFF}
\end{array}
\right.
\end{equation*}}
Although we include $\tilde{\zeta}_{3}(t)$ as an argument in (\ref{eq: h_ON}
) and (\ref{eq: h_OFF}) to stress the dependence on the stochastic process,
we will subsequently drop this for ease of notation as long as no confusion
arises. Hence, substituting (\ref{eq: h_ON}) and (\ref{eq: h_OFF}) into (\ref
{eq: HSC dynamics})-(\ref{eq: CRC dynamics}), yields

{\small \vspace{-3mm}
\begin{equation}
\dot{x}_{1}(t)=\left\{
\begin{array}{l}
\begin{array}{l}
\left\{ \alpha _{1}\left[ 1+\phi _{\alpha }^{ON}(t)\right] ^{-1}-\beta _{1}%
\left[ 1+\phi _{\beta }^{ON}(t)\right] ^{-1}\right. \\
\left. +m_{1}\left( \frac{h^{ON}\left( t\right) }{x_{3,0}}\right)
-(m_{1}+\lambda _{1})\right\} \cdot x_{1}(t) \\
+\mu _{1}+\zeta _{1}(t)\text{ \ \ \ \ \ \ \ \ \ \ \ \ \ \ \ \ \ \ \ \ \ \ \
\ if }q(t)=q^{ON}%
\end{array}
\\
\begin{array}{l}
\left\{ \alpha _{1}\left[ 1+\phi _{\alpha }^{OFF}(t)\right] ^{-1}-\beta _{1}%
\left[ 1+\phi _{\beta }^{OFF}(t)\right] ^{-1}\right. \\
\left. +m_{1}\left( \frac{h^{OFF}\left( t\right) }{x_{3,0}}\right)
-(m_{1}+\lambda _{1})\right\} \cdot x_{1}(t) \\
+\mu _{1}+\zeta _{1}(t)\text{ \ \ \ \ \ \ \ \ \ \ \ \ \ \ \ \ \ \ \ \ \ \ \
\ if }q(t)=q^{OFF}
\end{array}
\end{array}
\right.  \label{eq: x1 dynamics}
\end{equation}}

{\small \vspace{-3mm}
\begin{equation}
\dot{x}_{2}(t)=\left\{
\begin{array}{l}
\begin{array}{l}
\left[ \alpha _{2}\left( 1-d\frac{h^{ON}\left( t\right) }{x_{3,0}}\right)
-\beta _{2}\right] x_{2}(t) \\
+m_{1}\left( 1-\frac{h^{ON}\left( t\right) }{x_{3,0}}\right) x_{1}(t)+\zeta
_{2}(t) \\
\text{ \ \ \ \ \ \ \ \ \ \ \ \ \ \ \ \ \ \ \ \ \ \ \ \ \ \ \ \ if }
q(t)=q^{ON}
\end{array}
\\
\begin{array}{l}
\left[ \alpha _{2}\left( 1-d\frac{h^{OFF}\left( t\right) }{x_{3,0}}\right)
-\beta _{2}\right] x_{2}(t) \\
+m_{1}\left( 1-\frac{h^{OFF}\left( t\right) }{x_{3,0}}\right) x_{1}(t)+\zeta
_{2}(t) \\
\text{ \ \ \ \ \ \ \ \ \ \ \ \ \ \ \ \ \ \ \ \ \ \ \ \ \ \ \ if }q(t)=q^{OFF}
\end{array}
\end{array}
\right.  \label{eq: x2 dynamics}
\end{equation}}
with

{\small \vspace{-3mm}
\begin{align*}
\phi _{\alpha }^{ON}(t)& =e^{-\left( h^{ON}\left( t\right) -k_{1}\right)
k_{2}} \\
\phi _{\beta }^{ON}(t)& =e^{-\left( h^{ON}\left( t\right) -k_{3}\right)
k_{4}} \\
\phi _{\alpha }^{OFF}(t)& =e^{-\left( h^{OFF}\left( t\right) -k_{1}\right)
k_{2}} \\
\phi _{\beta }^{OFF}(t)& =e^{-\left( h^{OFF}\left( t\right) -k_{3}\right)
k_{4}}
\end{align*}}

The \emph{discrete (event-driven) dynamics} are dictated by the occurrence
of events that cause state transitions. Based on the event set $E=\left\{
e_{1},e_{2}\right\} $ we have defined, the occurrence of $e_{1}$ results in
a transition from $q^{ON}$ to $q^{OFF}$ and the occurrence of $e_{2}$
results in a transition from $q^{OFF}$ to $q^{ON}$.

\subsection{IAS Sensitivity Analysis}

Recall that the main goal of this work is to perform sensitivity analysis of (\ref
{eq: HSC dynamics})-(\ref{eq: z2 dynamics}) in order to identify critical model parameters and verify the extent to which the SHA model of prostate cancer evolution is robust to them. Of note, several potentially critical parameters
exist in the SHA model from \cite{FleckADHS2015}. This work is a first step towards analyzing their relative importance, in which we select a subset of all model parameters in order to illustrate the applicability of our IPA-based methodology. The parameters we consider here are $\alpha _{1}$ and $\alpha _{2}$ (HSC proliferation constant and CRC proliferation constant, respectively), as well as $\beta _{1}$ and $\beta _{2}$ (HSC
apoptosis constant and CRC apoptosis constant, respectively). These
constants are intrinsically related to the cancer cell subpopulations' net
growth rate, whose value dictates how fast the PSA threshold values will be
reached, and ultimately how soon treatment will be suspended or reinstated.
As a result, correctly estimating the values of $\alpha _{i}$ and $\beta
_{i} $, $i=1,2$, is presumably crucial for the purposes of personalized IAS
therapy design.

In this context, we define an extended parameter vector $\mathbf{\tilde{\theta }}=\left[
\tilde{\theta } _{1},\ldots ,\tilde{\theta } _{6}\right] $, where $\tilde{\theta } _{1}$ ($\tilde{\theta } _{2}$, respectively) corresponds to the lower (upper, respectively) threshold value of the patient's PSA\ level, $\tilde{\theta } _{3}$ ($\tilde{\theta } _{4}$, respectively) corresponds to the HSC (CRC, respectively) proliferation constant, and $\tilde{\theta } _{5}$ ($\tilde{\theta } _{6}$, respectively) corresponds to the HSC (CRC, respectively) apoptosis constant.

Within the SHA framework presented above, an IAS therapy can be
viewed as a controlled process $u\left( \tilde{\theta} ,t\right) $ characterized by
the parameter vector $\tilde{\theta} $, as in (\ref{eq: control set}), whose effect
can be quantified in terms of performance metrics of the form $J\left[
u\left( \tilde{\theta} ,t\right) \right] $. Of note, only the first two elements in vector $\tilde{\theta} $ are controllable, while the remaining parameters are not.

As in \cite{FleckADHS2015}, here we make use of a sample function defined in
terms of complementary measures of therapy success. In particular, we
consider the most adequate IAS treatment schemes to be those that $(i)$
ensure PSA levels are kept as low as possible; $(ii)$ reduce the frequency
of on and off-treatment cycles. From a practical perspective, $(i)$
translates into the ability to successfully keep the size of cancer cell
populations under control, which is directly influenced by the duration of
the on and off-treatment periods. On the other hand, $(ii)$ aims at reducing
the duration of on-treatment periods, thus decreasing the exposure of
patients to medication and their side effects, and consequently improving
the patients' quality of life throughout the treatment. Clearly
there is a trade-off between keeping tumor growth under control and the cost
associated with the corresponding IAS therapy. The latter is related to the
duration of the therapy and could potentially include fixed set up costs
incurred when therapy is reinstated. For simplicity, we disconsider fixed
set up costs and take $(ii)$ to be linearly proportional to the length of
the on-treatment cycles. Hence, we define our sample function as the sum of
the average PSA level and the average duration of an on-treatment cycle over
a fixed time interval $\left[ 0,T\right] $. We also take into account that
it may be desirable to design a therapy scheme which favors $(i)$ over $(ii)$
(or vice-versa) and thus associate weight $W$ with $(i)$ and $1-W$ with $%
(ii) $, where $0\leq W\leq 1$. Finally, to ensure that the trade-off between
$(i)$ and $(ii)$ is captured appropriately, we normalize our sample
function: we divide $(i)$ by the value of the patient's PSA level at the
start of the first on-treatment cycle ($PSA_{init}$), and normalize $(ii)$
by $T$.

Recall that the total population size of prostate cancer cells is assumed to
reflect the serum PSA concentration, and that we have defined clock
variables which measure the time elapsed in each of the treatment modes, so
that our sample function can be written as

{\small \vspace{-3mm}
\begin{equation}
L\left( \theta ,x(0),z(0),T\right) =%
\begin{array}[t]{l}
\frac{W}{T}\overset
{T}{\underset{0}{\int}}\left[ \frac{x_{1}\left( \theta ,t\right)
+x_{2}\left( \theta ,t\right) }{PSA_{init}}\right] dt \\
+\frac{(1-W)}{T}\overset
{T}{\underset{0}{\int}}\frac{z_{1}\left( t\right) }{T}dt%
\end{array}
\label{eq: sample function}
\end{equation}}
where $x(0)$ and $z(0)$ are given initial conditions. We can then define the
overall performance metric as

{\small \vspace{-3mm}
\begin{equation}
J\left( \tilde{\theta} ,x(0),z(0),T\right) =E\left[ L\left( \tilde{\theta}
,x(0),z(0),T\right) \right]  \label{eq: performance metric}
\end{equation}}

We note that it is not possible to derive a closed-form expression of $%
J\left( \tilde{\theta} ,x(0),z(0),T\right)$ without imposing limitations on the
processes $\{\zeta _{i}(t)\}$, $i=1,\ldots ,3$. Nevertheless, by assuming
only that $\zeta _{i}(t)$, $i=1,\ldots ,3$, are piecewise continuous w.p. 1,
we can successfully apply the IPA methodology developed for general SHS in
\cite{Cassandras2010} and obtain an estimate of $\nabla J\left( \tilde{\theta}
\right) $ by evaluating the sample gradient $\nabla L\left( \tilde{\theta} \right) $. We will assume that the
derivatives $dL\left( \tilde{\theta} \right) /d\tilde{\theta} _{i}$ exist w.p. 1 for all $
\tilde{\theta} _{i}\in \mathbb{R}^{+}$. It is also simple to verify that $L\left( \tilde{\theta} \right) $
is Lipschitz continuous for $\tilde{\theta} _{i}$ in $\mathbb{R}^{+}$. We will further assume that $\{\zeta _{i}(t)\}$, $i=1,\ldots ,3$, are stationary random processes over $\left[ 0,T\right] $ and that no two events can occur at the same time w.p. 1. Under these conditions, it has been shown in \cite{Cassandras2010} that $dL\left( \tilde{\theta} \right) /d\tilde{\theta} _{i}$ is an
\emph{unbiased} estimator of $dJ\left( \tilde{\theta} \right) /d\tilde{\theta} _{i}$, $i=1,\ldots ,6$. Hence, our goal is to compute the sample gradient $\nabla L\left( \tilde{\theta}
\right) $ using data extracted from a sample path of the system (e.g., by
simulating a sample path of our SHA\ model using clinical data), and use
this value as an estimate of $\nabla J\left( \tilde{\theta} \right) $.

\section{Infinitesimal Perturbation Analysis}
\label{IPA}

For completeness, we provide here a brief overview of the IPA
framework developed for stochastic hybrid systems in \cite{Cassandras2010}.
For such, we adopt a standard SHA definition \cite{Cassandras2008}:

{\small \vspace{-3mm}
\begin{equation}
G_{h}=\left( Q,X,E,U,f,\phi ,Inv,guard,\rho ,q_{0},x_{0}\right)
\label{eq: SHA}
\end{equation}}
where $Q$ is a set of discrete states; $X$ is a continuous state space; $E$
is a finite set of events; $U$ is a set of admissible controls; $f$ is a
vector field, $f:Q\times X\times U\rightarrow X$; $\phi $ is a discrete
state transition function, $\phi :Q\times X\times E\rightarrow Q$; $Inv$ is
a set defining an invariant condition (when this condition is violated at
some $q\in Q$, a transition must occur); $guard$ is a set defining a guard
condition, $guard\subseteq Q\times Q\times X$ (when this condition is
satisfied at some $q\in Q$, a transition is allowed to occur); $\rho $ is a
reset function, $\rho :Q\times Q\times X\times E\rightarrow X$; $q_{0}$ is
an initial discrete state; $x_{0}$ is an initial continuous state.

Consider a sample path of the system over $\left[ 0,T\right] $ and denote
the time of occurrence of the $k$th event (of any type) by $\tau _{k}\left(
\theta \right) $, where $\theta $ corresponds to the control parameter of
interest. Although we use the notation $\tau _{k}\left( \theta \right) $ to
stress the dependency of the event time on the control parameter, we will
subsequently use $\tau _{k}$ to indicate the time of occurrence of the $k$th
event where no confusion arises. In order to further simplify notation, we
shall denote the state and event time derivatives with respect to parameter $%
\theta $ as $x^{\prime }(t)\equiv \frac{\partial x(\theta ,t)}{\partial
\theta }$ and $\tau _{k}^{\prime }\equiv \frac{\partial \tau _{k}}{\partial
\theta }$, respectively, for $k=1,...,N$. Additionally, considering that the
system is at some discrete mode during an interval $\left[ \tau _{k},\tau
_{k+1}\right) $, we will denote its time-driven dynamics over such interval
as $f_{k}\left( x,\theta ,t\right) $. It is shown in \cite{Cassandras2010}
that the state derivative satisfies

{\small \vspace{-3mm}
\begin{equation}
\frac{d}{dt}x^{\prime }(t)=\frac{\partial f_{k}(t)}{\partial x}x^{\prime
}(t)+\frac{\partial f_{k}(t)}{\partial \theta }
\label{(eq): General state derivative}
\end{equation}}
with the following boundary condition:

{\small \vspace{-3mm}
\begin{equation}
x^{\prime }(\tau _{k}^{+})=x^{\prime }(\tau _{k}^{-})+\left[ f_{k-1}(\tau
_{k}^{-})-f_{k}(\tau _{k}^{+})\right] .\tau _{k}^{\prime }
\label{(eq): Boundary condition}
\end{equation}}
when $x(\theta ,t)$ is continuous in $t$ at $t=\tau _{k}$. Otherwise,

{\small \vspace{-3mm}
\begin{equation}
x^{\prime }(\tau _{k}^{+})=\frac{d\rho \left( q,q^{\prime },x,e\right) }{
d\theta }  \label{eq: Reset fn.}
\end{equation}}
where $\rho \left( q,q^{\prime },x,e\right) $ is the reset function defined
in (\ref{eq: SHA}).

Knowledge of $\tau _{k}^{\prime }$ is, therefore, needed in order to
evaluate (\ref{(eq): Boundary condition}). Following the framework in \cite
{Cassandras2010}, there are three types of events for a general stochastic hybrid system: \textit{(i) Exogenous event}. This type of event causes a discrete
state transition which is independent of parameter $\theta $ and, as a
result, $\tau _{k}^{\prime }=0$. \textit{(ii) Endogenous event}. In this case, there exists a
continuously differentiable function $g_{k}:\Re ^{n}\times \Theta
\rightarrow \Re $ such that $\tau _{k}=\min \left\{ t>\tau
_{k-1}:g_{k}\left( x(\theta ,t),\theta \right) =0\right\} $, which leads to
{\small
\begin{equation}
\tau _{k}^{\prime }=-\left[ \frac{\partial g_{k}}{\partial x}.f_{k-1}(\tau
_{k}^{-})\right] ^{-1}.\left( \frac{\partial g_{k}}{\partial \phi }+\frac{
\partial g_{k}}{\partial x}.x^{\prime }(\tau _{k}^{-})\right)
\label{(eq): Endogenous event}
\end{equation}}
where $\frac{\partial g_{k}}{\partial x}.f_{k-1}(\tau _{k}^{-})\neq 0$. \textit{(iii) Induced event}. Such an event is triggered by the
occurrence of another event at time $\tau _{m}\leq \tau _{k}$ and the
expression of $\tau _{k}^{\prime }$ depends on the event time derivative of
the triggering event ($\tau _{m}^{\prime }$) (details can be found in \cite
{Cassandras2010}).

Thus, IPA captures how changes in $\theta $ affect the event times and the
state of the system. Since interesting performance metrics are usually
expressed in terms of $\tau _{k}$ and $x(t)$, IPA can ultimately be used to
infer the effect that a perturbation in $\theta $ will have on such metrics. We end this overview by returning to our problem of personalized prostate cancer therapy design and thus defining the derivatives of the states $x_{n}(\tilde{\theta } ,t)$ and $z_{j}(\tilde{\theta } ,t)$ and
event times $\tau _{k}(\tilde{\theta } )$ with respect to $\tilde{\theta } _{i}$, $i=1,\ldots
,6$, $j=1,2$, $n=1,\ldots ,3$, as follows:

{\small \vspace{-3mm}
\begin{equation}
x_{n,i}^{\prime }(t)\equiv \frac{\partial x_{n}(\tilde{\theta},t)}{\partial
\tilde{\theta}_{i}}\text{, \ }z_{j,i}^{\prime }(t)\equiv \frac{\partial
z_{j}(\tilde{\theta},t)}{\partial \tilde{\theta}_{i}},\text{ }\tau
_{k,i}^{\prime }\equiv \frac{\partial \tau _{k}(\tilde{\theta})}{\partial
\tilde{\theta}_{i}}  \label{eq: Deriv def}
\end{equation}}
\vspace{-3mm}

In what follows, we derive the IPA state and event time derivatives for the
events identified in the SHA model of prostate cancer progression.

\subsection{State and Event Time Derivatives}

We proceed by analyzing the state evolution of our SHA\ model of prostate
cancer progression considering each of the states ($q^{ON}$ and $q^{OFF}$)
and events ($e_{1}$ and $e_{2}$) therein defined.

\emph{1. The system is in state }$q^{ON}$\emph{\ over interevent time
interval }$\left[ \tau _{k},\tau _{k+1}\right) $. Using (\ref{(eq): General
state derivative}) for $x_{1}\left( t\right) $, we obtain, for $i=1,\ldots ,6
$,

{\small \vspace{-3mm}
\begin{equation*}
\begin{array}{ll}
\frac{d}{dt}x_{1,i}^{\prime }(t) & =\frac{\partial f_{k}^{x_{1}}(t)}{%
\partial x_{1}}x_{1}^{\prime }(t)+\frac{\partial f_{k}^{x_{1}}(t)}{\partial
x_{2}}x_{2}^{\prime }(t) \\
& +\frac{\partial f_{k}^{x_{1}}(t)}{\partial z_{1}}z_{1}^{\prime }(t)+\frac{%
\partial f_{k}^{x_{1}}(t)}{\partial z_{2}}z_{2}^{\prime }(t)+\frac{\partial
f_{k}^{x_{1}}(t)}{\partial \tilde{\theta}_{i}}%
\end{array}%
\end{equation*}}
From (\ref{eq: x1 dynamics}), we have $\frac{\partial f_{k}^{x_{1}}(t)}{%
\partial x_{2}}=\frac{\partial f_{k}^{x_{1}}(t)}{\partial z_{j}}=\frac{%
\partial f_{k}^{x_{1}}(t)}{\partial \tilde{\theta}_{i}}=0$, $i=1,2,4,6$, $%
j=1,2$, and

{\small \vspace{-3mm}
\begin{eqnarray*}
&&%
\begin{array}{ll}
\frac{\partial f_{k}^{x_{1}}(t)}{\partial x_{1}} & =\alpha _{1}\left[ 1+\phi
_{\alpha }^{ON}(t)\right] ^{-1}-\beta _{1}\left[ 1+\phi _{\beta }^{ON}(t)%
\right] ^{-1} \\
& -m_{1}\left( 1-\frac{h^{ON}\left( t\right) }{x_{3,0}}\right) -\lambda _{1}%
\end{array}
\\
&&%
\begin{array}{ll}
\frac{\partial f_{k}^{x_{1}}(t)}{\partial \tilde{\theta}_{3}} & =x_{1}\left[
1+\phi _{\alpha }^{ON}(t)\right] ^{-1}%
\end{array}
\\
&&%
\begin{array}{ll}
\frac{\partial f_{k}^{x_{1}}(t)}{\partial \tilde{\theta}_{5}} & =-x_{1}\left[
1+\phi _{\beta }^{ON}(t)\right] ^{-1}%
\end{array}%
\end{eqnarray*}}
It is thus simple to verify that solving (\ref{(eq): General state
derivative}) for $x_{1,i}^{\prime }(t)$ yields, for $t\in \left[ \tau
_{k},\tau _{k+1}\right) $,

{\small \vspace{-3mm}
\begin{equation}
x_{1,i}^{\prime }(t)=x_{1,i}^{\prime }(\tau _{k}^{+})e^{A_{1}\left( t\right)
}\text{, \ }i=1,2,4,6  \label{eq: x1_prime ON}
\end{equation}}
{\small \vspace{-3mm}
\begin{equation}
x_{1,3}^{\prime }(t)=x_{1,3}^{\prime }(\tau _{k}^{+})e^{A_{1}\left( t\right)
}+A_{2}\left( t\right)   \label{eq: x13_prime ON}
\end{equation}}
{\small \vspace{-3mm}
\begin{equation}
x_{1,5}^{\prime }(t)=x_{1,5}^{\prime }(\tau _{k}^{+})e^{A_{1}\left( t\right)
}+A_{3}\left( t\right)   \label{eq: x15_prime ON}
\end{equation}}
with

{\small \vspace{-3mm}
\begin{eqnarray}
&&%
\begin{array}{l}
A_{1}\left( t\right) \equiv \int_{\tau _{k}}^{t}\left[ \frac{\alpha _{1}}{%
1+\phi _{\alpha }^{ON}(t)}-\frac{\beta _{1}}{1+\phi _{\beta }^{ON}(t)}\right]
dt \\
\text{ \ }-\int_{\tau _{k}}^{t}\frac{m_{1}}{x_{3,0}}h^{ON}\left( t\right)
dt-\left( m_{1}+\lambda _{1}\right) \left( t-\tau _{k}\right)
\end{array}
\label{eq: A(t)} \\
&&A_{2}\left( t\right) \equiv e^{A_{1}\left( t\right) }\int_{\tau _{k}}^{t}
\left[ \frac{x_{1}\left( t\right) }{1+\phi _{\alpha }^{ON}(t)}%
e^{-A_{1}\left( t\right) }\right] dt \\
&&A_{3}\left( t\right) \equiv e^{A_{1}\left( t\right) }\int_{\tau _{k}}^{t}
\left[ -\frac{x_{1}\left( t\right) }{1+\phi _{\beta }^{ON}(t)}e^{A_{1}\left(
t\right) }\right] dt
\end{eqnarray}}
In particular, at $\tau _{k+1}^{-}$:

{\small
\begin{equation}
x_{1,i}^{\prime }(\tau _{k+1}^{-})=x_{1,i}^{\prime }(\tau
_{k}^{+})e^{A\left( \tau _{k}\right) }  \label{eq: Deriv x1 t1}
\end{equation}}
{\small \vspace{-3mm}
\begin{equation}
x_{1,3}^{\prime }(\tau _{k+1}^{-})=x_{1,3}^{\prime }(\tau
_{k}^{+})e^{A\left( \tau _{k}\right) }+A_{2}\left( \tau _{k}\right)
\label{eq: Deriv x13 t1}
\end{equation}}
{\small \vspace{-3mm}
\begin{equation}
x_{1,5}^{\prime }(\tau _{k+1}^{-})=x_{1,5}^{\prime }(\tau
_{k}^{+})e^{A\left( \tau _{k}\right) }+A_{3}\left( \tau _{k}\right)
\label{eq: Deriv x15 t1}
\end{equation}}
where {\small $A_{1}\left( \tau _{k}\right) $}, {\small $A_{2}\left( \tau _{k}\right) $}, and {\small $
A_{3}\left( \tau _{k}\right) $} are given from (\ref{eq: A(t)}).

Similarly for $x_{2}\left( t\right) $, we have from (\ref{eq: x2 dynamics})
that $\frac{\partial f_{k}^{x_{2}}(t)}{\partial z_{j}}=\frac{\partial
f_{k}^{x_{2}}(t)}{\partial \tilde{\theta}_{i}}=0$, $i=1,2,3,5$, $j=1,2$, and

{\small \vspace{-3mm}
\begin{equation*}
\begin{array}{l}
\frac{\partial f_{k}^{x_{2}}(t)}{\partial x_{1}}=m_{1}\left( 1-\frac{%
h^{ON}\left( t\right) }{x_{3,0}}\right)  \\
\frac{\partial f_{k}^{x_{2}}(t)}{\partial x_{2}}=\alpha _{2}\left( 1-d\frac{%
h^{ON}\left( t\right) }{x_{3,0}}\right) -\beta _{2} \\
\frac{\partial f_{k}^{x_{2}}(t)}{\partial \tilde{\theta}_{4}}=\left( 1-d%
\frac{h^{ON}\left( t\right) }{x_{3,0}}\right) x_{2}\left( t\right)  \\
\frac{\partial f_{k}^{x_{2}}(t)}{\partial \tilde{\theta}_{6}}=-x_{2}\left(
t\right)
\end{array}%
\end{equation*}}
Combining the last four equations and solving for $x_{2,i}^{\prime }(t)$
yields, for $t\in \left[ \tau _{k},\tau _{k+1}\right) $,

{\small \vspace{-3mm}
\begin{equation}
x_{2,i}^{\prime }(t)=x_{2,i}^{\prime }(\tau
_{k}^{+})e^{B_{1}(t)}+B_{2}\left( t,x_{1,i}^{\prime }(\tau
_{k}^{+}),A_{1}\left( t\right) \right) \text{, \ \ }i=1,2,3,5
\label{eq: x2_prime ON}
\end{equation}}
{\small \vspace{-3mm}
\begin{equation}
x_{2,4}^{\prime }(t)=x_{2,4}^{\prime }(\tau
_{k}^{+})e^{B_{1}(t)}+B_{3}\left( t,x_{1,4}^{\prime }(\tau
_{k}^{+}),B_{1}\left( t\right) \right)   \label{eq: x24_prime ON}
\end{equation}}
{\small \vspace{-3mm}
\begin{equation}
x_{2,6}^{\prime }(t)=x_{2,6}^{\prime }(\tau
_{k}^{+})e^{B_{1}(t)}+B_{4}\left( t,x_{1,6}^{\prime }(\tau
_{k}^{+}),B_{1}\left( t\right) \right)   \label{eq: x26_prime ON}
\end{equation}}
with%

{\small \vspace{-3mm}
\begin{align}
& B_{1}\left( t\right) \equiv \int_{\tau _{k}}^{t}\left[ \alpha _{2}\left(
1-d\frac{h^{ON}\left( t\right) }{x_{3,0}}\right) -\beta _{2}\right] dt
\label{eq: B(t)} \\
& B_{2}\left( \cdot \right) \equiv e^{B_{1}(t)}\int_{\tau
_{k}}^{t}G_{1}\left( t,\tau _{k}\right) e^{-B_{1}(t)}dt  \notag \\
& B_{3}\left( \cdot \right) \equiv e^{B_{1}(t)}\int_{\tau
_{k}}^{t}G_{2}\left( t,\tau _{k}\right) e^{-B_{1}(t)}dt \\
& B_{4}\left( \cdot \right) \equiv e^{B_{1}(t)}\int_{\tau
_{k}}^{t}G_{3}\left( t,\tau _{k}\right) e^{-B_{1}(t)}dt
\end{align}}
where {\small $G_{1}\left( t,\tau _{k}\right) =m_{1}\left( 1-\frac{%
h^{ON}\left( t\right) }{x_{3,0}}\right) x_{1,i}^{\prime }(\tau
_{k}^{+})e^{A_{1}\left( t\right) }$}, {\small $G_{2}\left( t,\tau
_{k}\right) =e^{-B_{1}(t)}x_{2}\left( t\right) \left( 1-d\frac{h^{ON}\left(
t\right) }{x_{3,0}}\right) $}$+e^{-B_{1}(t)}x_{1,4}^{\prime }(t)\cdot
m_{1}\left( 1-d\frac{h^{ON}\left( t\right) }{x_{3,0}}\right) $, {\small $%
G_{3}\left( t,\tau _{k}\right) =e^{-B_{1}(t)}\left[ x_{1,6}^{\prime
}(t)\cdot m_{1}\left( 1-d\frac{h^{ON}\left( t\right) }{x_{3,0}}\right)
-x_{2}\left( t\right) \right] $}, $t\in \left[ \tau _{k},\tau _{k+1}\right) $%
.

In particular, at $\tau _{k+1}^{-}$:

{\small
\begin{equation}
x_{2,i}^{\prime }(\tau _{k+1}^{-})=x_{2,i}^{\prime }(\tau
_{k}^{+})e^{B_{1}(\tau _{k})}+B_{2}\left( \tau _{k},x_{1,i}^{\prime }(\tau
_{k}^{+}),A\left( \tau _{k}\right) \right)   \label{eq: Deriv x2 t1}
\end{equation}}
{\small \vspace{-3mm}
\begin{equation}
x_{2,4}^{\prime }(\tau _{k+1}^{-})=x_{2,4}^{\prime }(\tau
_{k}^{+})e^{B_{1}(\tau _{k})}+B_{3}\left( t,x_{1,4}^{\prime }(\tau
_{k}^{+}),B_{1}\left( \tau _{k}\right) \right)   \label{eq: Deriv x24 t1}
\end{equation}}
{\small \vspace{-3mm}
\begin{equation}
x_{2,6}^{\prime }(\tau _{k+1}^{-})=x_{2,6}^{\prime }(\tau
_{k}^{+})e^{B_{1}(\tau _{k})}+B_{4}\left( t,x_{1,6}^{\prime }(\tau
_{k}^{+}),B_{1}\left( \tau _{k}\right) \right)   \label{eq: Deriv x26 t1}
\end{equation}}
where {\small $B_{1}\left( \tau _{k}\right) $}, {\small $B_{2}\left( \tau
_{k},x_{1,i}^{\prime }(\tau _{k}^{+}),A\left( \tau _{k}\right) \right) $}, {\small $%
B_{3}\left( t,x_{1,4}^{\prime }(\tau _{k}^{+}),B_{1}\left( \tau _{k}\right)
\right) $}, and {\small $B_{4}\left( t,x_{1,6}^{\prime }(\tau _{k}^{+}),B_{1}\left(
\tau _{k}\right) \right) $} are given from (\ref{eq: B(t)}).

Finally, for the \textquotedblleft clock\textquotedblright\ state variable,
from (\ref{eq: z1 dynamics})-(\ref{eq: z2 dynamics}) we have $\frac{\partial
f_{k}^{z_{i}}(t)}{\partial x_{n}}=\frac{\partial f_{k}^{z_{i}}(t)}{\partial
z_{j}}=\frac{\partial f_{k}^{z_{i}}(t)}{\partial \tilde{\theta}_{i}}=0$, $%
n,j=1,2$, $i=1,\ldots ,6$, so that $\frac{d}{dt}z_{j,i}^{\prime }(t)=0$, $%
j=1,2$, $i=1,\ldots ,6$, for $t\in \left[ \tau _{k},\tau _{k+1}\right) $.
Hence, $z_{j,i}^{\prime }(t)=z_{j,i}^{\prime }(\tau _{k}^{+})$, $j=1,2$, $%
i=1,\ldots ,6$, and $t\in \left[ \tau _{k},\tau _{k+1}\right) $.

\emph{2. The system is in state }$q^{OFF}$\emph{\ over interevent time
interval }$\left[ \tau _{k},\tau _{k+1}\right) $. Starting with $x_{1}\left(
t\right) $, based on (\ref{eq: x1 dynamics}) we once again have $\frac{%
\partial f_{k}^{x_{1}}(t)}{\partial x_{2}}=\frac{\partial f_{k}^{x_{1}}(t)}{%
\partial zj}=\frac{\partial f_{k}^{x_{1}}(t)}{\partial \tilde{\theta}_{i}}=0$%
, $j=1,2$, $i=1,2,4,6$, but now

{\small \vspace{-3mm}
\begin{eqnarray*}
&&%
\begin{array}{l}
\frac{\partial f_{k}^{x_{1}}(t)}{\partial x_{1}}=\alpha _{1}\left[ 1+\phi
_{\alpha }^{OFF}(t)\right] ^{-1}-\beta _{1}\left[ 1+\phi _{\beta }^{OFF}(t)%
\right] ^{-1} \\
\text{ \ \ \ \ \ \ \ \ }-m_{1}\left( 1-\frac{h^{OFF}\left( t\right) }{x_{3,0}%
}\right) -\lambda _{1}%
\end{array}
\\
&&%
\begin{array}{ll}
\frac{\partial f_{k}^{x_{1}}(t)}{\partial \tilde{\theta}_{3}} & =x_{1}\left[
1+\phi _{\alpha }^{ON}(t)\right] ^{-1}%
\end{array}
\\
&&%
\begin{array}{ll}
\frac{\partial f_{k}^{x_{1}}(t)}{\partial \tilde{\theta}_{5}} & =-x_{1}\left[
1+\phi _{\beta }^{ON}(t)\right] ^{-1}%
\end{array}%
\end{eqnarray*}}
Therefore, (\ref{(eq): General state derivative}) implies that, for $t\in %
\left[ \tau _{k},\tau _{k+1}\right) $:

{\small \vspace{-3mm}
\begin{equation}
x_{1,i}^{\prime }(t)=x_{1,i}^{\prime }(\tau _{k}^{+})e^{C_{1}\left( t\right)
}\text{, \ }i=1,2,4,6  \label{eq: x1_prime OFF}
\end{equation}}
{\small \vspace{-3mm}
\begin{equation}
x_{1,3}^{\prime }(t)=x_{1,3}^{\prime }(\tau _{k}^{+})e^{C_{1}\left( t\right)
}+C_{2}\left( t\right)   \label{eq: x13_prime OFF}
\end{equation}}
{\small \vspace{-3mm}
\begin{equation}
x_{1,5}^{\prime }(\tau _{k+1}^{-})=x_{1,5}^{\prime }(\tau
_{k}^{+})e^{C_{1}\left( \tau _{k}\right) }+C_{3}\left( t\right)
\label{eq: x15_prime OFF}
\end{equation}}
with

{\small \vspace{-3mm}
\begin{eqnarray}
&&%
\begin{array}{l}
C_{1}\left( t\right) \equiv \int_{\tau _{k}}^{t}\left[ \frac{\alpha _{1}}{%
1+\phi _{\alpha }^{OFF}(t)}-\frac{\beta _{1}}{1+\phi _{\beta }^{OFF}(t)}%
\right] dt \\
\text{ }-\int_{\tau _{k}}^{t}\frac{m_{1}}{x_{3,0}}h^{OFF}\left( t\right)
dt-\left( m_{1}+\lambda _{1}\right) \left( t-\tau _{k}\right)
\end{array}
\label{eq: C(t)} \\
&&C_{2}\left( t\right) \equiv e^{C_{1}\left( t\right) }\int_{\tau _{k}}^{t}
\left[ \frac{x_{1}\left( t\right) }{1+\phi _{\alpha }^{OFF}(t)}%
e^{-C_{1}\left( t\right) }\right] dt \\
&&C_{3}\left( t\right) \equiv e^{C_{1}\left( \tau _{k}\right) }\int_{\tau
_{k}}^{\tau _{k+1}}\left[ -\frac{x_{1}\left( t\right) }{1+\phi _{\beta
}^{OFF}(t)}e^{C_{1}\left( t\right) }\right] dt
\end{eqnarray}}
In particular, at $\tau _{k+1}^{-}$:

{\small
\begin{equation}
x_{1,i}^{\prime }(\tau _{k+1}^{-})=x_{1,i}^{\prime }(\tau
_{k}^{+})e^{C_{1}\left( \tau _{k}\right) }  \label{eq: Deriv x1 t3minus}
\end{equation}}
{\small \vspace{-3mm}
\begin{equation}
x_{1,3}^{\prime }(\tau _{k+1}^{-})=x_{1,3}^{\prime }(\tau
_{k}^{+})e^{C_{1}\left( \tau _{k}\right) }+C_{2}\left( \tau _{k}\right)
\label{eq: Deriv x13 t3minus}
\end{equation}}
{\small \vspace{-3mm}
\begin{equation}
x_{1,5}^{\prime }(\tau _{k+1}^{-})=x_{1,5}^{\prime }(\tau
_{k}^{+})e^{C_{1}\left( \tau _{k}\right) }+C_{3}\left( \tau _{k}\right)
\label{eq: Deriv x15 t3minus}
\end{equation}}
where {\small $C_{1}\left( \tau _{k}\right) $}, {\small $C_{2}\left( \tau _{k}\right) $}, and {\small $%
C_{3}\left( \tau _{k}\right) $} are given from (\ref{eq: C(t)}).

Similarly for $x_{2}(t)$, we have

{\small \vspace{-3mm}
\begin{equation*}
\begin{array}{l}
\frac{\partial f_{k}^{x_{2}}(t)}{\partial x_{1}}=m_{1}\left( 1-\frac{%
h^{OFF}\left( t\right) }{x_{3,0}}\right)  \\
\frac{\partial f_{k}^{x_{2}}(t)}{\partial x_{2}}=\alpha _{2}\left( 1-d\frac{%
h^{OFF}\left( t\right) }{x_{3,0}}\right) -\beta _{2} \\
\frac{\partial f_{k}^{x_{2}}(t)}{\partial \tilde{\theta}_{4}}=\left( 1-d%
\frac{h^{OFF}\left( t\right) }{x_{3,0}}\right) x_{2}\left( t\right)  \\
\frac{\partial f_{k}^{x_{2}}(t)}{\partial \tilde{\theta}_{6}}=-x_{2}\left(
t\right)
\end{array}%
\end{equation*}}
It is thus straightforward to verify that (\ref{(eq): General state
derivative}) yields, for $t\in \left[ \tau _{k},\tau _{k+1}\right) $,

{\small \vspace{-3mm}
\begin{equation}
x_{2,i}^{\prime }(t)=x_{2,i}^{\prime }(\tau
_{k}^{+})e^{D_{1}(t)}+D_{2}\left( t,x_{1,i}^{\prime }(\tau
_{k}^{+}),C_{1}\left( t\right) \right) \text{, \ \ }i=1,2,4,6
\label{eq: x2_prime OFF}
\end{equation}}
{\small \vspace{-3mm}
\begin{equation}
x_{2,4}^{\prime }(t)=x_{2,4}^{\prime }(\tau
_{k}^{+})e^{D_{1}(t)}+D_{3}\left( t,x_{1,4}^{\prime }(\tau
_{k}^{+}),D_{1}\left( t\right) \right)   \label{eq: x24_prime OFF}
\end{equation}}
{\small \vspace{-3mm}
\begin{equation}
x_{2,6}^{\prime }(t)=x_{2,6}^{\prime }(\tau
_{k}^{+})e^{D_{1}(t)}+D_{4}\left( t,x_{1,6}^{\prime }(\tau
_{k}^{+}),D_{1}\left( t\right) \right)   \label{eq: x26_prime OFF}
\end{equation}}
with

{\small \vspace{-3mm}
\begin{align}
& D_{1}\left( t\right) \equiv \int_{\tau _{k}}^{t}\left[ \alpha _{2}\left(
1-d\frac{h^{OFF}\left( t\right) }{x_{3,0}}\right) -\beta _{2}\right] dt
\label{eq: D(t)} \\
& D_{2}\left( \cdot \right) \equiv e^{D_{1}(t)}\int_{\tau
_{k}}^{t}G_{2}\left( t,\tau _{k}\right) e^{-D_{1}(t)}dt  \notag \\
& D_{3}\left( \cdot \right) \equiv e^{D_{1}(t)}\int_{\tau
_{k}}^{t}G_{3}\left( t,\tau _{k}\right) e^{-D_{1}(t)}dt \\
& D_{4}\left( \cdot \right) \equiv e^{D_{1}(t)}\int_{\tau
_{k}}^{t}G_{4}\left( t,\tau _{k}\right) e^{-D_{1}(t)}dt
\end{align}}
where {\small $G_{2}\left( t,\tau _{k}\right) =m_{1}\left( 1-\frac{h^{OFF}\left(
t\right) }{x_{3,0}}\right) x_{1,i}^{\prime }(\tau _{k}^{+})e^{C_{1}\left(
t\right) }$}, {\small $G_{3}\left( t,\tau _{k}\right) =x_{2}\left( t\right) \left( 1-d%
\frac{h^{OFF}\left( t\right) }{x_{3,0}}\right) +x_{1,4}^{\prime }(t)\cdot
m_{1}\left( 1-d\frac{h^{OFF}\left( t\right) }{x_{3,0}}\right) $}, {\small $%
G_{4}\left( t,\tau _{k}\right) =x_{1,6}^{\prime }(t)\cdot m_{1}\left( 1-d%
\frac{h^{OFF}\left( t\right) }{x_{3,0}}\right) -x_{2}\left( t\right) $}, $%
t\in \left[ \tau _{k},\tau _{k+1}\right) $.

In particular, at $\tau _{k+1}^{-}$:

{\small
\begin{equation}
x_{2,i}^{\prime }(\tau _{k+1}^{-})=x_{2,i}^{\prime }(\tau
_{k}^{+})e^{D_{1}(\tau _{k})}+D_{2}\left( \tau _{k},x_{1,i}^{\prime }(\tau
_{k}^{+}),C\left( \tau _{k}\right) \right)   \label{eq: Deriv x2 tkminus}
\end{equation}}
{\small \vspace{-3mm}
\begin{equation}
x_{2,4}^{\prime }(t)=x_{2,4}^{\prime }(\tau
_{k}^{+})e^{D_{1}(t)}+D_{3}\left( t,x_{1,4}^{\prime }(\tau
_{k}^{+}),D_{1}\left( \tau _{k}\right) \right)
\label{eq: Deriv x24 tkminus}
\end{equation}}
{\small \vspace{-3mm}
\begin{equation}
x_{2,6}^{\prime }(t)=x_{2,6}^{\prime }(\tau
_{k}^{+})e^{D_{1}(t)}+D_{4}\left( t,x_{1,6}^{\prime }(\tau
_{k}^{+}),D_{1}\left( \tau _{k}\right) \right)
\label{eq: Deriv x26 tkminus}
\end{equation}}
where {\small $D_{1}\left( \tau _{k}\right) $}, {\small $D_{2}\left( \cdot \right) $}, {\small $%
D_{3}\left( \cdot \right) $}, and {\small $D_{4}\left( \cdot \right) $} are given from
(\ref{eq: D(t)}).

Finally, for the \textquotedblleft clock\textquotedblright\ state variable,
based on (\ref{eq: z1 dynamics})-(\ref{eq: z2 dynamics}) we once again have $%
\frac{\partial f_{k}^{z_{i}}(t)}{\partial x_{n}}=\frac{\partial
f_{k}^{z_{i}}(t)}{\partial z_{j}}=\frac{\partial f_{k}^{z_{i}}(t)}{\partial
\tilde{\theta}_{i}}=0$, $n,j=1,2$, $i=1,\ldots ,6$, so that $\frac{d}{dt}%
z_{j,i}^{\prime }(t)=0$, $j=1,2$, $i=1,\ldots ,6$, for $t\in \left[ \tau
_{k},\tau _{k+1}\right) $. As a result, $z_{j,i}^{\prime
}(t)=z_{j,i}^{\prime }(\tau _{k}^{+})$, $j=1,2$, $i=1,\ldots ,6$, and $t\in %
\left[ \tau _{k},\tau _{k+1}\right) $.

\emph{3. A state transition from }$q^{ON}$\emph{\ to }$q^{OFF}$\emph{\
occurs at time }$\tau _{k}$. This necessarily implies that event $e_{1}$
took place at time $\tau _{k}$, i.e., $q(t)=q^{ON}$, $t\in \left[ \tau
_{k-1},\tau _{k}\right) $ and $q(t)=q^{OFF}$, $t\in \left[ \tau _{k},\tau
_{k+1}\right) $. From (\ref{(eq): Boundary condition}) we have, for $%
i=1,\ldots ,6$,

{\small \vspace{-3mm}
\begin{equation}
x_{1,i}^{\prime }(\tau _{k}^{+})=x_{1,i}^{\prime }(\tau _{k}^{-})+\left[
f_{k}^{x_{1}}(\tau _{k}^{-})-f_{k+1}^{x_{1}}(\tau _{k}^{+})\right] \cdot
\tau _{k,i}^{\prime }  \label{eq: x1_prime ON/OFF}
\end{equation}}
and

{\small \vspace{-3mm}
\begin{equation}
x_{2,i}^{\prime }(\tau _{k}^{+})=x_{2,i}^{\prime }(\tau _{k}^{-})+\left[
f_{k}^{x_{2}}(\tau _{k}^{-})-f_{k+1}^{x_{2}}(\tau _{k}^{+})\right] \cdot
\tau _{k,i}^{\prime }  \label{eq: x2_prime ON/OFF}
\end{equation}}
where $f_{k}^{x_{1}}(\tau _{k}^{-})-f_{k+1}^{x_{1}}(\tau _{k}^{+})$ and $%
f_{k}^{x_{2}}(\tau _{k}^{-})-f_{k+1}^{x_{2}}(\tau _{k}^{+})$ ultimately
depend on $h^{ON}\left( \tau _{k}^{-}\right) $ and $h^{OFF}\left( \tau
_{k}^{+}\right) $. Evaluating $h^{ON}\left( \tau _{k}^{-}\right) $ from (\ref%
{eq: h_ON}) over the appropriate time interval results in

{\small \vspace{-3mm}
\begin{equation*}
\begin{array}{ll}
h^{ON}\left( \tau _{k}^{-}\right)  & =x_{3}(\tau _{k-1}^{+})e^{-\left( \tau
_{k}-\tau _{k-1}\right) /\sigma } \\
& +\mu _{3}\sigma \lbrack 1-e^{-\left( \tau _{k}-\tau _{k-1}\right) /\sigma
}]+\tilde{\zeta}_{3}(\tau _{k})%
\end{array}%
\end{equation*}}
and it follows directly from (\ref{eq: h_OFF}) that $h^{OFF}\left( \tau
_{k}^{+}\right) =x_{3}(\tau _{k}^{+})$. Moreover, by continuity of $x_{n}(t)$
(due to conservation of mass), $x_{n}(\tau _{k}^{+})=x_{n}(\tau _{k}^{-})$, $%
n=1,2$. Also, since we have assumed that $\left\{ \zeta _{i}(t)\right\} $, $%
i=1,\ldots ,3$, is piecewise continuous w.p.1 and that no two events can
occur at the same time w.p.1, $\zeta _{i}(\tau _{k}^{-})=$ $\zeta _{i}(\tau
_{k}^{+})$, $i=1,\ldots ,3$. Hence, for $x_{1}(t)$, evaluating $\Delta
_{f}^{1}\left( \tau _{k}\right) \equiv f_{k}^{x_{1}}(\tau
_{k}^{-})-f_{k+1}^{x_{1}}(\tau _{k}^{+})$ yields

{\small \vspace{-3mm}
\begin{equation}
\begin{array}{l}
\Delta _{f}^{1}\left( \tau _{k},\zeta _{3}\left( \tau _{k}\right) \right)
=\left\{ \alpha _{1}\left[ 1+\phi _{\alpha }^{ON}(\tau _{k}^{-})\right]
^{-1}\right.  \\
\text{ }-\alpha _{1}\left[ 1+\phi _{\alpha }^{OFF}(\tau _{k}^{+})\right]
^{-1}-\beta _{1}\left[ 1+\phi _{\beta }^{ON}(\tau _{k}^{-})\right] ^{-1} \\
\text{ }+\beta _{1}\left[ 1+\phi _{\beta }^{OFF}(\tau _{k}^{+})\right] ^{-1}
\\
\text{ }\left. +\frac{m_{1}}{x_{3,0}}\left[ h^{ON}\left( \tau
_{k}^{-}\right) -x_{3}(\tau _{k})\right] \right\} \cdot x_{1}(\tau _{k})%
\end{array}
\label{eq: delta f1 on/off}
\end{equation}}
Finally, the term $\tau _{k,i}^{\prime }$, which corresponds to the event
time derivative with respect to $\tilde{\theta}_{i}$ at event time $\tau _{k}
$, is determined using (\ref{(eq): Endogenous event}), as detailed in (\ref%
{eq: Lemma 1}) later.

A similar analysis applies to $x_{2}(t)$, so that $f_{k}^{x_{2}}(\tau
_{k}^{-})$ and $f_{k+1}^{x_{2}}(\tau _{k}^{+})$ ultimately depend on $%
h^{ON}\left( \tau _{k}^{-}\right) $ and $h^{OFF}\left( \tau _{k}^{+}\right) $%
, respectively. Hence, evaluating $\Delta _{f}^{2}\left( \tau _{k}\right)
\equiv f_{k}^{x_{2}}(\tau _{k}^{-})-f_{k+1}^{x_{2}}(\tau _{k}^{+})$ from (%
\ref{eq: x2 dynamics}) yields

{\small \vspace{-3mm}
\begin{equation}
\begin{array}{ll}
\Delta _{f}^{2}\left( \tau _{k},\zeta _{3}\left( \tau _{k}\right) \right)  &
=\frac{\alpha _{2}d}{x_{3,0}}\left[ x_{3}(\tau _{k})-h^{ON}\left( \tau
_{k}^{-}\right) \right] \cdot x_{2}(\tau _{k}) \\
& -\frac{m_{1}}{x_{3,0}}\left[ h^{ON}\left( \tau _{k}^{-}\right) -x_{3}(\tau
_{k})\right] \cdot x_{1}(\tau _{k})%
\end{array}
\label{eq: deltaF2 ON/OFF}
\end{equation}}

In the case of the \textquotedblleft clock\textquotedblright\ state
variable, $z_{1}(t)$ is discontinuous in $t$ at $t=\tau _{k}$, while $%
z_{2}(t)$ is continuous. Hence, based on (\ref{eq: Reset fn.}) and (\ref{eq:
z1 dynamics}), we have that $z_{1,i}^{\prime }(\tau _{k}^{+})=0$. From (\ref%
{(eq): Boundary condition}) and (\ref{eq: z2 dynamics}), it is
straightforward to verify that $z_{2,i}^{\prime }(\tau
_{k}^{+})=z_{2,i}^{\prime }(\tau _{k}^{-})-\tau _{k,i}^{\prime }$, $%
i=1,\ldots ,6$.

\emph{4. A state transition from }$q^{OFF}$\emph{\ to }$q^{ON}$\emph{\
occurs at time }$\tau _{k}$. This necessarily implies that event $e_{2}$
took place at time $\tau _{k}$, i.e., $q(t)=q^{OFF}$, $t\in \left[ \tau
_{k-1},\tau _{k}\right) $ and $q(t)=q^{ON}$, $t\in \left[ \tau _{k},\tau
_{k+1}\right) $. The same reasoning as above holds, so that (\ref{eq:
x1_prime ON/OFF})-(\ref{eq: x2_prime ON/OFF}) also apply. For $x_{1}(t)$, $%
f_{k}^{x_{1}}(\tau _{k}^{-})-f_{k+1}^{x_{1}}(\tau _{k}^{+})$ can be
evaluated from (\ref{eq: x1 dynamics}) and ultimately depends on $%
h^{OFF}\left( \tau _{k}^{-}\right) $ and $h^{ON}\left( \tau _{k}^{+}\right) $%
. Evaluating $h^{OFF}\left( \tau _{k}^{-}\right) $ from (\ref{eq: h_OFF})
over the appropriate time interval results in

{\small \vspace{-3mm}
\begin{equation*}
\begin{array}{ll}
h^{OFF}\left( \tau _{k}^{-}\right)  & =x_{3}(\tau _{k-1}^{+})e^{-\left( \tau
_{k}-\tau _{k-1}\right) /\sigma } \\
& +(\mu _{3}\sigma +x_{3,0})[1-e^{-\left( \tau _{k}-\tau _{k-1}\right)
/\sigma }]+\tilde{\zeta}_{3}(\tau _{k})%
\end{array}%
\end{equation*}}
and it follows directly from (\ref{eq: h_ON}) that $h^{ON}\left( \tau
_{k}^{+}\right) =x_{3}(\tau _{k}^{+})$.

As in the previous case, continuity due to conservation of mass applies, so
that evaluating $\Delta _{f}^{1}(\tau _{k})\equiv f_{k}^{x_{1}}(\tau
_{k}^{-})-f_{k+1}^{x_{1}}(\tau _{k}^{+})$ yields

{\small \vspace{-3mm}
\begin{equation}
\begin{array}{l}
\Delta _{f}^{1}(\tau _{k},\zeta _{3}\left( \tau _{k}\right) )=\left\{ \alpha
_{1}\left[ 1+\phi _{\alpha }^{OFF}(\tau _{k}^{-})\right] ^{-1}\right.  \\
\text{ }-\alpha _{1}\left[ 1+\phi _{\alpha }^{ON}(\tau _{k}^{+})\right]
^{-1}-\beta _{1}\left[ 1+\phi _{\beta }^{OFF}(\tau _{k}^{-})\right] ^{-1} \\
\text{ }+\beta _{1}\left[ 1+\phi _{\beta }^{ON}(\tau _{k}^{+})\right] ^{-1}
\\
\text{ }\left. +\frac{m_{1}}{x_{3,0}}\left[ h^{OFF}\left( \tau
_{k}^{-}\right) -x_{3}(\tau _{k})\right] \right\} \cdot x_{1}(\tau _{k})%
\end{array}
\label{eq: deltaF1 OFF/ON}
\end{equation}}

Similarly for $x_{2}(t)$, by evaluating $\Delta _{f}^{2}(\tau _{k})\equiv
f_{k}^{x_{2}}(\tau _{k}^{-})-f_{k+1}^{x_{2}}(\tau _{k}^{+})$ from (\ref{eq:
x2 dynamics}), and making the appropriate simplifications due to continuity,
we obtain

{\small \vspace{-3mm}
\begin{equation}
\begin{array}{ll}
\Delta _{f}^{2}(\tau _{k},\zeta _{3}\left( \tau _{k}\right) ) & =\frac{%
\alpha _{2}d}{x_{3,0}}\left[ x_{3}(\tau _{k})-h^{OFF}\left( \tau
_{k}^{-}\right) \right] \cdot x_{2}(\tau _{k}) \\
& -\frac{m_{1}}{x_{3,0}}\left[ h^{OFF}\left( \tau _{k}^{-}\right)
-x_{3}(\tau _{k})\right] \cdot x_{1}(\tau _{k})%
\end{array}
\label{eq: deltaF2 OFF/ON}
\end{equation}}

In the case of the \textquotedblleft clock\textquotedblright\ state
variable, $z_{1}(t)$ is continuous in $t$ at $t=\tau _{k}$, while $z_{2}(t)$
is discontinuous. As a result, based on (\ref{(eq): Boundary condition}) and
(\ref{eq: z1 dynamics}), we have that $z_{1,i}^{\prime }(\tau
_{k}^{+})=z_{1,i}^{\prime }(\tau _{k}^{-})-\tau _{k,i}^{\prime }$. From (\ref%
{eq: Reset fn.}) and (\ref{eq: z2 dynamics}), it is simple to verify that $%
z_{2,i}^{\prime }(\tau _{k}^{+})=0$, $i=1,\ldots ,6$.

Note that, since $z_{j,i}^{\prime }(t)=z_{j,i}^{\prime }(\tau _{k}^{+})$, $%
t\in \left[ \tau _{k},\tau _{k+1}\right) $, we will have that $%
z_{j,i}^{\prime }(\tau _{k}^{-})=z_{j,i}^{\prime }(\tau _{k-1}^{+})$, $j=1,2$%
, $i=1,\ldots ,6$. Moreover, the sample path of our SHA consists of a
sequence of alternating $e_{1}$ and $e_{2}$ events, which implies that $%
z_{1,i}^{\prime }(\tau _{k}^{-})=0$ if event $e_{1}$ occurred at $\tau
_{k-1} $, while $z_{2,i}^{\prime }(\tau _{k}^{-})=0$ if event $e_{2}$
occurred at $\tau _{k-1}$. Then, adopting the notation $p,\overline{p}%
=\left\{ 1,2\right\} $ such that $p+\overline{p}=3$, we have:

{\small
\begin{equation}
z_{p,i}^{\prime }(\tau _{k}^{+})=\left\{
\begin{array}{ll}
-\tau _{k,i}^{\prime } & \text{if event }e_{\overline{p}}\text{ occurs at }%
\tau _{k} \\
0 & \text{otherwise}%
\end{array}%
\right.  \label{eq: State deriv. z1}
\end{equation}}

We now proceed with a general result which applies to all events defined for
our SHA\ model. We denote the time of occurrence of the $j$th state
transition by $\tau _{j}$, define its derivative with respect to the control
parameters as $\tau _{j,i}^{\prime }\equiv \frac{\partial \tau _{j}}{%
\partial \tilde{\theta}_{i}}$, $i=1,\ldots ,6$, and also define $%
f_{j}^{x_{n}}\left( \tau _{j}\right) \equiv \dot{x}_{n}(\tau _{j})$, $n=1,2$.

\begin{lemma}
When an event $e_{p}$, $p=1,2$, occurs, the derivative $\tau _{j,i}^{\prime }
$, $i=1,\ldots ,6$, of state transition times $\tau _{j}$, $j=1,2,\ldots $
with respect to the control parameters $\tilde{\theta}_{i}$, $i=1,\ldots ,6$%
, satisfies:%
\begin{equation}
\tau _{j,i}^{\prime }=\left\{
\begin{array}{ll}
\frac{\mathbf{1}-x_{1,i}^{\prime }(\tau _{j}^{-})-x_{2,i}^{\prime }(\tau
_{j}^{-})}{f_{j-1}^{x_{1}}(\tau _{j}^{-})+f_{j-1}^{x_{2}}(\tau _{j}^{-})} &
\begin{array}{l}
\text{if }e_{1}\text{ occurs and }i=1 \\
\text{or }e_{2}\text{ occurs and }i=2%
\end{array}
\\
&  \\
\frac{-x_{1,i}^{\prime }(\tau _{j}^{-})-x_{2,i}^{\prime }(\tau _{j}^{-})}{%
f_{j-1}^{x_{1}}(\tau _{j}^{-})+f_{j-1}^{x_{2}}(\tau _{j}^{-})} &
\begin{array}{l}
\text{if }e_{1}\text{ occurs and }i \neq 1 \\
\text{or }e_{2}\text{ occurs and }i \neq 2%
\end{array}%
\end{array}%
\right.   \label{eq: Lemma 1}
\end{equation}
\end{lemma}

\begin{proof} The proof is omitted here, but can be found in \cite{FleckADHS2015}.
\end{proof}

We note that the numerator in (\ref{eq: Lemma 1}) is determined using (\ref%
{eq: Deriv x1 t1}) and (\ref{eq: Deriv x2 t1}) if $q(\tau _{j}^{-})=q^{ON}$,
or (\ref{eq: Deriv x1 t3minus}) and (\ref{eq: Deriv x2 tkminus}) if $q(\tau
_{j}^{-})=q^{OFF}$. Moreover, the denominator in (\ref{eq: Lemma 1}) is
computed using (\ref{eq: x1 dynamics})-(\ref{eq: x2 dynamics}) and it is
simple to verify that, if event $e_{1}$ takes place at time $\tau _{j}$,

{\small \vspace{-3mm}
\begin{equation}
\begin{array}{l}
f_{j-1}^{x_{1}}(\tau _{j}^{-})+f_{j-1}^{x_{2}}(\tau _{j}^{-})=\alpha _{1}
\left[ 1+\phi _{\alpha }^{ON}(\tau _{j}^{-})\right] ^{-1}\cdot x_{1}(\tau
_{j}) \\
\text{ \ }-\left\{ \beta _{1}\left[ 1+\phi _{\beta }^{ON}(\tau _{j}^{-})%
\right] ^{-1}+\lambda _{1}\right\} \cdot x_{1}(\tau _{j})+\mu _{1} \\
\text{ \ }+\left[ \alpha _{2}\left( 1-d\frac{h^{ON}\left( \tau
_{j}^{-}\right) }{x_{3,0}}\right) -\beta _{2}\right] \cdot x_{2}(\tau _{j})
\\
\text{ \ \ }+\zeta _{1}(\tau _{j})+\zeta _{2}(\tau _{j})%
\end{array}
\label{eq: event e1}
\end{equation}}
and, if event $e_{2}$ takes place at time $\tau _{j}$,

{\small \vspace{-3mm}
\begin{equation}
\begin{array}{l}
f_{j-1}^{x_{1}}(\tau _{j}^{-})+f_{j-1}^{x_{2}}(\tau _{j}^{-})=\alpha _{1}
\left[ 1+\phi _{\alpha }^{OFF}(\tau _{j}^{-})\right] ^{-1}\cdot x_{1}(\tau
_{j}) \\
\text{ \ \ }-\left\{ \beta _{1}\left[ 1+\phi _{\beta }^{OFF}(\tau _{j}^{-})%
\right] ^{-1}+\lambda _{1}\right\} \cdot x_{1}(\tau _{j})+\mu _{1} \\
\text{ \ }+\left[ \alpha _{2}\left( 1-d\frac{h^{OFF}\left( \tau
_{j}^{-}\right) }{x_{3,0}}\right) -\beta _{2}\right] \cdot x_{2}(\tau _{j})
\\
\text{ \ \ }+\zeta _{1}(\tau _{j})+\zeta _{2}(\tau _{j})%
\end{array}
\label{eq: event e2}
\end{equation}}
\vspace{-3mm}

\subsection{Cost Derivative}

Let us denote the total number of on and off-treatment periods (complete or
incomplete) in $\left[ 0,T\right] $ by $K_{T}$. Also let $\xi _{k}$ denote
the start of the $k^{th}$ period and $\eta _{k}$ denote the end of the $%
k^{th}$ period (of either type). Finally, let $M_{T}=\lfloor \frac{K_{T}}{2}%
\rfloor $ be the total number of complete on-treatment periods, and $\Delta
_{m}^{ON}$ denote the duration of the $m^{th}$ complete on-treatment period,
where clearly

\begin{equation*}
\Delta _{m}^{ON}\equiv \eta _{m}-\xi _{m}\text{, }m=1,2,\ldots
\end{equation*}
\vspace{-3mm}

It was shown in \cite{FleckADHS2015} that the derivative of the sample function $L(\tilde{\theta} )$ with respect to the
control parameters satisfies:

\begin{equation}
\begin{array}{l}
\frac{dL(\tilde{\theta} )}{d\tilde{\theta} _{i}}=\frac{W}{T}\overset{K_{T}}{\underset{k=1}{%
\sum }}\int_{\xi _{k}}^{\eta _{k}}\left[ \frac{x_{1,i}^{\prime }(\tilde{\theta}
,t)+x_{2,i}^{\prime }(\tilde{\theta} ,t)}{PSA_{init}}\right] dt \\
\text{ \ \ }+\frac{\left( 1-W\right) }{T}\overset{M_{T}}{\underset{m=1}{\sum
}}\frac{\Delta _{m}^{ON}}{T}\cdot \left( \eta _{m,i}^{\prime }-\xi
_{m,i}^{\prime }\right) \\
\text{ \ \ }-\frac{\left( 1-W\right) }{T}\mathbf{1}\left[ K_{T}\text{ is odd}%
\right] \cdot \xi _{M_{T}+1,i}^{\prime }\cdot \left( \frac{T-\xi _{M_{T}+1}}{%
T}\right)%
\end{array}
\label{eq: Theorem 1}
\end{equation}
where $\mathbf{1}\left[ \cdot \right] $ is the usual indicator function and $%
PSA_{init}$ is the value of the patient's PSA level at the start of the
first on-treatment cycle.

The derivation of (\ref{eq: Theorem 1}) is omitted here, but can be found in \cite{FleckADHS2015}. We now proceed to present the results obtained from our IPA-driven
sensitivity analysis.

\section{Results}
\label{Results}

The results shown here represent an initial study of sensitivity analysis
applied to a SHA\ model of prostate cancer progression in which we consider
only noise and fluctuations associated with cell population dynamics, and do
not account for noise in the patient's androgen level. Representing
randomness as Gaussian white noise, the authors in \cite{Tanaka2010}
verified that variable time courses of the PSA\ levels were produced without
losing the tendency of the deterministic system, thus yielding simulation
results that were comparable to the statistics of clinical data. For this
reason, in this work we take $\left\{ \zeta _{i}\left( t\right) \right\} $, $%
i=1,2$, to be Gaussian white noise with zero mean and standard deviation of
0.001, similarly to \cite{Tanaka2010}, although we remind the reader that
our methodology applies independently of the distribution chosen to
represent $\left\{ \zeta _{i}\left( t\right) \right\} $, $i=1,2$. We
estimate the noise associated with cell population dynamics at event times
by randomly sampling from a uniform distribution with zero mean and standard
deviation of 0.001. Simulations of the prostate cancer model as a pure DES
are thus run to generate sample path data to which the IPA estimator is
applied. In all results reported here, we measure the sample path length in
terms of the number of days elapsed since the onset of IAS therapy, which we
choose to be $T=2500$ days.

Three sets of simulations were performed: in the first one we consider the optimal therapy configuration determined for Patient \#15 in \cite{FleckNAHS2016} and vary the values of $\tilde{\theta}_{i}$, $i=3,\ldots,6 $ (one at a time). For the second, we use PSA threshold values that yield a therapy of maximum cost and once again vary the values of $\tilde{\theta}_{i}$, $i=3,\ldots,6 $ (one at a time). Finally, in our third set of simulations, we let $\tilde{\theta}_{i}$, $i=3,\ldots,6 $, take the nominal values from \cite{Liu2015} and vary the values of $\tilde{\theta}_{1}$ and $\tilde{\theta}_{2}$ along their allowable ranges.

Table \ref{tb: Results opt} presents the sensitivity of the model
parameters, $\frac{dL}{d\tilde{\theta}_{i}}$, $i=3,\ldots ,6$, around the
optimal configuration $\left[ \tilde{\theta}_{1}^{\ast },\tilde{\theta}%
_{2}^{\ast }\right] =\left[ 1.5,8.0\right] $ for the values of $\tilde{\theta%
}_{i}$, $i=3,\ldots ,6$, fitted to the model of Patient \#15 in \cite%
{Liu2015}. We note that the results shown here are representative of the phenomena that may be uncovered by this type of analysis, and were hence generated using the model of a single patient. Moreover, while the use of different patient models may potentially reveal additional phenomena, the insights presented below are interesting in their own right and thus set the stage for extending this analysis to other patients.

\vspace{-3mm}
\begin{table}[htb] \centering%
\caption{Sensitivity of model parameters around the optimal therapy
configuration}\label{tb: Results opt}%
\begin{tabular}{|c|c|c|c|}
\hline
$\frac{dL}{d\tilde{\theta}_{3}}$ & $\frac{dL}{d\tilde{\theta}_{4}}$ & $\frac{%
dL}{d\tilde{\theta}_{5}}$ & $\frac{dL}{d\tilde{\theta}_{6}}$ \\ \hline
$5.44$ & $-0.25$ & $-5.95$ & $0.28$ \\ \hline
\end{tabular}%
\end{table}%
\vspace{-3mm}

Recall that $\tilde{\theta}_{3}$ and $\tilde{\theta}_{4}$ correspond to the
HSC proliferation constant and CRC proliferation constant, respectively,
while $\tilde{\theta}_{5}$ and $\tilde{\theta}_{6}$ are the HSC apoptosis
constant and CRC apoptosis constant, respectively. Several interesting
remarks can be made based on the above results; in what follows, we adopt
the notation $x\approx y$ to indicate that $x$ takes values \emph{%
approximately equal} to $y$.

From Table \ref{tb: Results opt}, it can be seen that $\frac{dL}{d\tilde{%
\theta}_{3}}\approx -\frac{dL}{d\tilde{\theta}_{5}}$ and $\frac{dL}{d\tilde{%
\theta}_{4}}\approx -\frac{dL}{d\tilde{\theta}_{6}}$, which indicates that
the sensitivities of proliferation and apoptosis constants are of the same
order of magnitude (in absolute value) for any given cancer cell
subpopulation. It is also possible to verify a large difference in the
values of the sensitivities across different subpopulations; in fact the
sensitivities of HSC proliferation and apoptosis constants are approximately
21 times higher than those of CRC constants. In other words, the system is
more sensitive to changes in the HSC constants than changes in the CRC
constants, i.e., $\tilde{\theta}_{3}$ and $\tilde{\theta}_{5}$ are more
critical model parameters than $\tilde{\theta}_{4}$ and $\tilde{\theta}_{6}$%
. Additionally, $\frac{dL}{d\tilde{\theta}_{3}}>0$ and $\frac{dL}{d\tilde{%
\theta}_{6}}>0$, while $\frac{dL}{d\tilde{\theta}_{4}}<0$ and $\frac{dL}{d%
\tilde{\theta}_{5}}>0$. A possible explanation for this has to do with the
fact that HSCs are the dominant subpopulation in a prostate tumor under IAS
therapy, which means that the size of this subpopulation has a greater
impact on the overall size of the tumor, and consequently, on the value of
the PSA level. As a result, increasing  $\tilde{\theta}_{3}$ (or decreasing $%
\tilde{\theta}_{5}$) leads to an increase in the size of the HSC population,
reflected in the PSA level, thus increasing the overall cost. On the other
hand, increasing $\tilde{\theta}_{4}$ (or decreasing $\tilde{\theta}_{6}$)
directly increases the size of the CRC population; however, since the
conditions under which CRCs thrive are those under which HSCs perish, an
increase in the size of the CRC population implies that the size of the HSC
population will decrease. Given that HSCs are the dominant subpopulation,
the PSA level would ultimately decrease, thus decreasing the overall cost.

The effect of changes in $\tilde{\theta}_{i}$, $i=3,\ldots ,6$, on the
sensitivity of model parameters was analyzed next. As the values of $\tilde{%
\theta}_{i}$, $i=3,\ldots ,6$, were progressively altered, two scenarios
emerged: $Scenario$ $A$ - a set of model parameter values was found for which the
evolution of the prostate tumor is permanently halted after one or two
cycles of treatment, i.e., the simulated IAS therapy scheme is curative; $Scenario$ $B$ - a set of model parameter values was found for which the prostate tumor
grows in an uncontrollable manner, i.e., the simulated IAS therapy scheme is
ineffective. $Scenario$ $A$ occurred when $\tilde{\theta}_{3}$ took on
values that were at least 15\% smaller than the nominal value given in \cite%
{Liu2015}, or when $\tilde{\theta}_{5}$ took on values that were at least
30\% smaller than the nominal value given in \cite{Liu2015}; no variations
in either $\tilde{\theta}_{4}$ or $\tilde{\theta}_{6}$ lead to such
scenario. On the other hand, $Scenario$ $B$ occurred when $\tilde{\theta}%
_{3}$ took on values that were at least 15\% higher than the nominal value
given in \cite{Liu2015}, or when $\tilde{\theta}_{4}$ took on values that
were at least 10\% higher than the nominal value given in \cite{Liu2015}, or
when $\tilde{\theta}_{5}$ took on values that were at least 30\% higher than
the nominal value given in \cite{Liu2015}, or when $\tilde{\theta}_{6}$ took
on values that were at least 10\% smaller than the nominal value given in
\cite{Liu2015}.

In practical terms, the above results indicate that if the optimal IAS\
therapy (designed using the model of Patient \#15) were applied to a new patient
whose HSC population dynamics are slower than those of Patient \#15 (i.e.,
the new patient's HSC proliferation constant is at least 15\% smaller than
that of Patient \#15; or the new patient's HSC apoptosis constant is at
least 30\% smaller than that of Patient \#15), then the size of the new
patient's tumor would remain stable and under control after at most two
treatment cycles. On the other hand, if the optimal IAS\ therapy (designed
using the model of Patient \#15) were applied to a new patient whose HSC
population dynamics are faster than those of Patient \#15, then the size of
the new patient's tumor would grow uncontrollably.

In our second set of simulations, we let $\tilde{\theta}_{1}$ and $\tilde{%
\theta}_{2}$ take suboptimal values and once again vary the values of $%
\tilde{\theta}_{i}$, $i=3,\ldots ,6$ (one at a time). Table \ref{tb: Results
subopt} presents the sensitivity of the model parameters, $\frac{dL}{d\tilde{%
\theta}_{i}}$, $i=3,\ldots ,6$, around the suboptimal configuration $\left[
\tilde{\theta}_{1},\tilde{\theta}_{2}\right] =\left[ 7.5,15.0\right] $ for
the values of $\tilde{\theta}_{1}$, $i=3,\ldots ,6$, fitted to the model of
Patient \#15 in \cite{Liu2015}.

\vspace{-3mm}
\begin{table}[htb] \centering%
\caption{Sensitivity of model parameters around a suboptimal therapy
configuration}\label{tb: Results subopt}%
\begin{tabular}{|c|c|c|c|}
\hline
$\frac{dL}{d\tilde{\theta}_{3}}$ & $\frac{dL}{d\tilde{\theta}_{4}}$ & $\frac{%
dL}{d\tilde{\theta}_{5}}$ & $\frac{dL}{d\tilde{\theta}_{6}}$ \\ \hline
$17.78$ & $0.014$ & $-17.15$ & $-0.016$ \\ \hline
\end{tabular}%
\end{table}%
\vspace{-3mm}

Once again, the effect of changes in $\tilde{\theta}_{i}$, $i=3,\ldots ,6$,
on the sensitivity of model parameters was analyzed. $Scenario$ $A$ occurred
when $\tilde{\theta}_{3}$ took on values that were at least 10\% smaller
than the nominal value given in \cite{Liu2015}, or when $\tilde{\theta}_{5}$
took on values that were at least 20\% larger than the nominal value given
in \cite{Liu2015}; no variations in either $\tilde{\theta}_{4}$ or $\tilde{%
\theta}_{6}$ lead to such scenario. Moreover, $Scenario$ $B$ did not emerge
in any of the simulations performed under this suboptimal configuration.

In our third set of simulations, we investigate the behavior of the model
parameter sensitivities, $\frac{dL}{d\tilde{\theta}_{i}}$, $i=3,\ldots ,6$,
across different PSA threshold settings. In particular, we study how the
sensitivity values change as we move from an optimal therapy setting towards
various suboptimal settings. For such, we let $\tilde{\theta}_{i}$, $%
i=3,\ldots ,6$, take the nominal values given in \cite{Liu2015} and vary the
values of the lower and upper PSA thresholds along $\left[ \tilde{\theta}%
_{1}^{\min },\tilde{\theta}_{1}^{\max }\right] $ and $\left[ \tilde{\theta}%
_{2}^{\min },\tilde{\theta}_{2}^{\max }\right] $, respectively.

Figs. \ref{fig: sensTH3_pat15}-\ref{fig: sensTH6_pat15} show how the values of the sensitivities, $\frac{dL}{d\tilde{\theta}_{i}}$, $i=3,\ldots ,6$,
vary as a function of the values of the lower and upper PSA thresholds ($\frac{dL}{d\tilde{\theta}_{1}}$ and $\frac{dL}{d\tilde{\theta}_{2}}$, respectively) for the model of Patient \#15.

\begin{figure}
[tbh]
\begin{center}
\includegraphics[natheight=3.69in,
natwidth=4.80in,
height=1.85in,
width=2.4in]{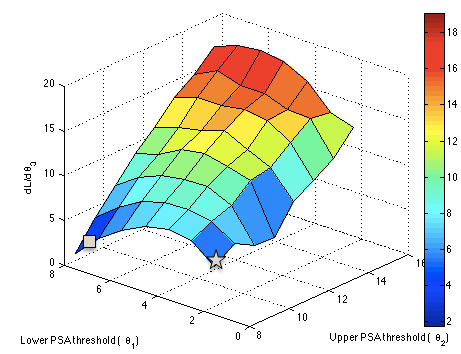}
\caption{Sensitivity of $\tilde{\theta}_{3}$ as a function of the values of $\tilde{\theta}_{1}$ and $\tilde{\theta}_{2}$ (Patient \#15)}
\label{fig: sensTH3_pat15}
\end{center}
\end{figure}

\begin{figure}
[tbh]
\begin{center}
\includegraphics[natheight=3.79in,
natwidth=4.92in,
height=1.9in,
width=2.46in]{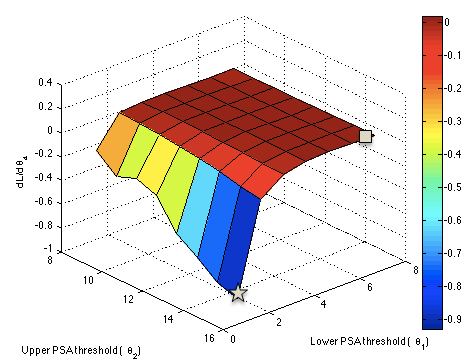}
\caption{Sensitivity of $\tilde{\theta}_{4}$ as a function of the values of $\tilde{\theta}_{1}$ and $\tilde{\theta}_{2}$ (Patient \#15)}
\label{fig: sensTH4_pat15}
\end{center}
\end{figure}

\begin{figure}
[tbh]
\begin{center}
\includegraphics[natheight=3.91in,
natwidth=4.83in,
height=1.96in,
width=2.42in]{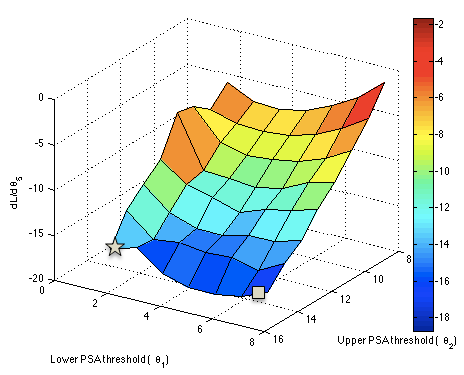}
\caption{Sensitivity of $\tilde{\theta}_{5}$ as a function of the values of $\tilde{\theta}_{1}$ and $\tilde{\theta}_{2}$ (Patient \#15)}
\label{fig: sensTH5_pat15}
\end{center}
\end{figure}

\begin{figure}
[tbh]
\begin{center}
\includegraphics[natheight=3.74in,
natwidth=4.91in,
height=1.87in,
width=2.46in]{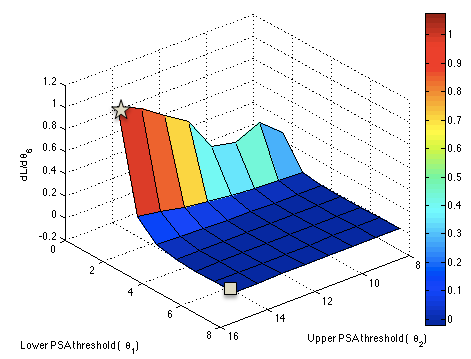}
\caption{Sensitivity of $\tilde{\theta}_{6}$ as a function of the values of $\tilde{\theta}_{1}$ and $\tilde{\theta}_{2}$ (Patient \#15)}
\label{fig: sensTH6_pat15}
\end{center}
\end{figure}

Figs. \ref{fig: sensTH3_pat1}-\ref{fig: sensTH6_pat1} show how the values of the sensitivities, $\frac{dL}{d\tilde{\theta}_{i}}$, $i=3,\ldots ,6$,
vary as a function of the values of the lower and upper PSA thresholds ($\frac{dL}{d\tilde{\theta}_{1}}$ and $\frac{dL}{d\tilde{\theta}_{2}}$, respectively) for the model of Patient \#1.

\begin{figure}
[tbh]
\begin{center}
\includegraphics[natheight=3.9in,
natwidth=4.92in,
height=1.95in,
width=2.46in]{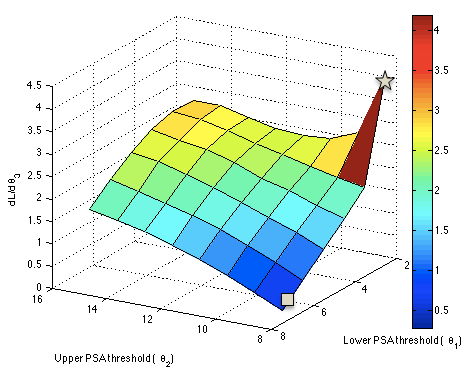}
\caption{Sensitivity of $\tilde{\theta}_{3}$ as a function of the values of $\tilde{\theta}_{1}$ and $\tilde{\theta}_{2}$ (Patient \#1)}
\label{fig: sensTH3_pat1}
\end{center}
\end{figure}

\begin{figure}
[tbh]
\begin{center}
\includegraphics[natheight=3.83in,
natwidth=5.17in,
height=1.9in,
width=2.59in]{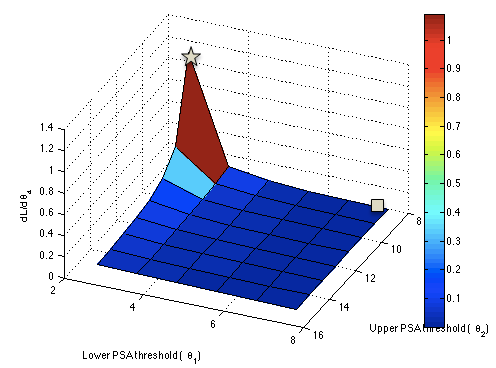}
\caption{Sensitivity of $\tilde{\theta}_{4}$ as a function of the values of $\tilde{\theta}_{1}$ and $\tilde{\theta}_{2}$ (Patient \#1)}
\label{fig: sensTH4_pat1}
\end{center}
\end{figure}

\begin{figure}
[tbh]
\begin{center}
\includegraphics[natheight=3.9in,
natwidth=4.81in,
height=1.95in,
width=2.41in]{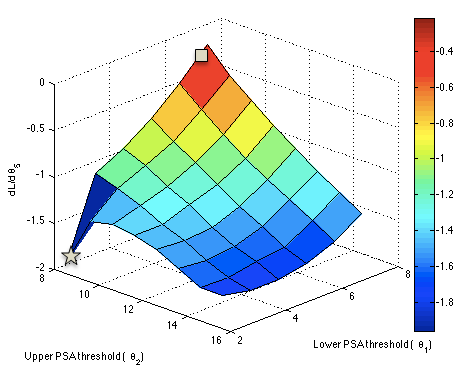}
\caption{Sensitivity of $\tilde{\theta}_{5}$ as a function of the values of $\tilde{\theta}_{1}$ and $\tilde{\theta}_{2}$ (Patient \#1)}
\label{fig: sensTH5_pat1}
\end{center}
\end{figure}

\begin{figure}
[tbh]
\begin{center}
\includegraphics[natheight=3.72in,
natwidth=4.97in,
height=1.86in,
width=2.49in]{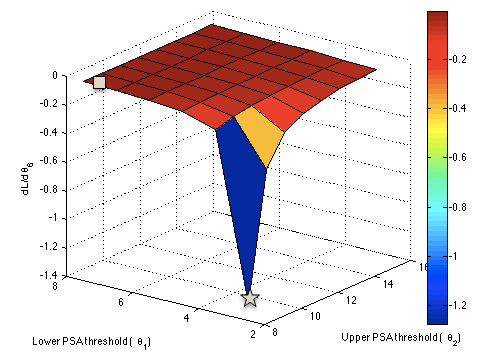}
\caption{Sensitivity of $\tilde{\theta}_{6}$ as a function of the values of $\tilde{\theta}_{1}$ and $\tilde{\theta}_{2}$ (Patient \#1)}
\label{fig: sensTH6_pat1}
\end{center}
\end{figure}

The above results lend themselves to the following discussion: first, the values of the model parameter sensitivities, $\frac{dL}{d\tilde{\theta}_{i}}$, $i=3,\ldots ,6$, are neither monotonically increasing nor monotonically decreasing along the allowable ranges of $\frac{dL}{d\tilde{\theta}_{1}}$ and $\frac{dL}{d\tilde{\theta}_{2}}$; this is verified for both patients. Second, the system is more sensitive to parameters $\tilde{\theta}_{3}$ and $\tilde{\theta}_{5}$ (HSC proliferation and apoptosis constants, respectively), and more robust to $\tilde{\theta}_{4}$ and $\tilde{\theta}_{6}$ (CRC proliferation and apoptosis constants, respectively); again this is verified across different patients. A possible explanation for this has to do with the fact that HSCs are commonly assumed to be the dominant subpopulation is a prostate tumor undergoing IAS therapy, which means that the size of the this subpopulation has a greater impact on the overall size of the tumor and, consequently, on the value of the PSA level.

Additionally, note that two points are marked in Figs. \ref{fig: sensTH3_pat15}-\ref{fig: sensTH6_pat1}: a star marks
the optimal therapy configuration and a square marks the values of $\tilde{%
\theta}_{1}$ and $\tilde{\theta}_{2}$ for which the sensitivities $\frac{dL}{%
d\tilde{\theta}_{i}}$, $i=3,\ldots ,6$, are minimal. In \cite{FleckNAHS2016} the optimal therapy configurations were found to be $\left[
\tilde{\theta}_{1}^{\ast },\tilde{\theta}_{2}^{\ast }\right] =\left[ 1.5,8.0%
\right] $ for Patient \#15 and $\left[ \tilde{\theta}_{1}^{\ast },\tilde{%
\theta}_{2}^{\ast }\right] =\left[ 2.5,8.0\right] $ for Patient \#1. As it
can be seen in Figs. \ref{fig: sensTH3_pat15}-\ref{fig: sensTH6_pat1}, these settings are not located in the regions
of minimum sensitivities. Of note, the sensitivities $\frac{dL}{d\tilde{%
\theta}_{i}}$, $i=3,\ldots ,6$, take their minimum value at the same
suboptimal configuration (namely $\left[ \tilde{\theta}_{1},\tilde{\theta}%
_{2}\right] =\left[ 7.5,8.0\right] $) across different patients. This could
potentially point to the existence of an underlying, and most likely as of
yet poorly understood, equilibrium of cancer cell subpopulation dynamics at
this suboptimal setting.

Moreover, the tradeoff between system fragility and optimality seems more
strongly applicable to $\tilde{\theta}_{1}$, and less so to $\tilde{\theta}%
_{2}$; interestingly, the value of $\tilde{\theta}_{1}^{\ast }$ differed
across patients, while $\tilde{\theta}_{2}^{\ast }$ did not. In this sense,
relaxing the optimality condition in favor of increased system robustness
could potentially be worthwhile in at least some cases. In fact, for Patient
\#1, moving from an optimal therapy setting to a slightly suboptimal setting
along $\tilde{\theta}_{2}$ (namely $\left[ \tilde{\theta}_{1},\tilde{\theta}%
_{2}\right] =\left[ 2.5,9.0\right] $) leads to a 9\% increase in the cost of
treatment. However, the model parameter sensitivities at this setting
decrease by approximately 30\% for $\tilde{\theta}_{3}$ and $\tilde{\theta}%
_{5}$ and by approximately 70\% for $\tilde{\theta}_{4}$ and $\tilde{\theta}%
_{6}$. If we move to a suboptimal setting along $\tilde{\theta}_{1}$ (namely
$\left[ \tilde{\theta}_{1},\tilde{\theta}_{2}\right] =\left[ 3.5,8.0\right] $%
), the cost increases by 16\%, while the sensitivities decrease by
approximately 50\% for $\tilde{\theta}_{3}$ and $\tilde{\theta}_{5}$ and by
approximately 90\% for $\tilde{\theta}_{4}$ and $\tilde{\theta}_{6}$. In
this case, it seems advantageous to tradeoff optimality
for increased robustness.

It is interesting to note that the above analysis is not consistently
verified across different patients. In fact, for Patient \#15, a marked
decrease in system fragility only occurs when we move to a suboptimal
setting along $\tilde{\theta}_{1}$ (namely $\left[ \tilde{\theta}_{1},\tilde{%
\theta}_{2}\right] =\left[ 7.5,8.0\right] $), at which point the
sensitivities decrease by approximately 70\% for $\tilde{\theta}_{3}$ and $%
\tilde{\theta}_{5}$ and by approximately 99\% for $\tilde{\theta}_{4}$ and $%
\tilde{\theta}_{6}$. However, there is an increase in the cost value of the
order of 70\%, which indicates that system optimality is significantly
compromised. These results highlight the importance of applying our methodology on a patient-by-patient basis. More generally, they validate recent efforts favoring the development of \emph{personalized} cancer therapies, as opposed to traditional treatment schemes that are typically generated over a cohort of patients and thus effective only on average.

\section{Conclusions}
\label{Conclusions}
We use a stochastic model of prostate cancer evolution under IAS therapy to perform sensitivity analysis with respect to several important model parameters. We find the system to be more sensitive to changes in the HSC proliferation and apoptosis constants than changes in the CRC proliferation and apoptosis constants. We also identify a set of model parameter values for which the simulated IAS therapy scheme is essentially curative, as well as a set of model parameters for which the prostate tumor grows in an uncontrollable manner. Finally, we verify that relaxing optimality in favor of increased system stability can potentially be of interest in at least some cases.

This work is a first attempt at investigating the tradeoff between optimality and system robustness/fragility in stochastic models of cancer evolution. A subset of all model parameters is selected and a case study of prostate cancer is used to illustrate the applicability of our IPA-based methodology. Nevertheless, there exist several other  potentially critical parameters in the SHA model of prostate cancer evolution we study, so that part of our ongoing work includes extending this sensitivity analysis study to other model parameters. Additionally, future work includes applying this methodology to other types of cancer (e.g., breast cancer), as well as other diseases that are known to progress in stages (e.g., tuberculosis).
\vspace{-1mm}
\bibliographystyle{is-plain}
\bibliography{CHAbib}

\begin{thebibliography}{10}

\bibitem{Bruchovsky2006}
N.~Bruchovsky, L.~Klotz, J.~Crook, S.~Malone, C.~Ludgte, W.~Morris, M.E.
  Gleave, S.L. Goldenberg, and P.S. Rennie.
\newblock Final results of the {C}anadian prospective phase {II} trial of
  intermittent androgen suppression for men in biochemical recurrence after
  radiotherapy for locally advanced prostate cancer.
\newblock {\em Cancer}, 107:\penalty0 389--395, 2006.

\bibitem{Bruchovsky2007}
N.~Bruchovsky, L.~Klotz, J.~Crook, S.~Malone, C.~Ludgte, W.~Morris, M.E.
  Gleave, S.L. Goldenberg, and P.S. Rennie.
\newblock Locally advanced prostate cancer biochemical results from a
  prospective phase {II} study of intermittent androgen suppression for men
  with evidence of prostate-specific antigen recurrence after radiotherapy.
\newblock {\em Cancer}, 109:\penalty0 858--867, 2007.

\bibitem{Cassandras2008}
C.G. Cassandras and S.~Lafortune.
\newblock {\em Introduction to Discrete Event Systems}.
\newblock Springer, 2nd edition, 2008.

\bibitem{Cassandras2010}
C.G. Cassandras, Y.~Wardi, C.G. Panayiotou, and C.~Yao.
\newblock Perturbation analysis and optimization of stochastic hybrid systems.
\newblock {\em European Journal of Control}, 6\penalty0 (6):\penalty0 642--664,
  2010.

\bibitem{FleckADHS2015}
J.~L. Fleck and C.~G. Cassandras.
\newblock Infinitesimal perturbation analysis for personalized cancer therapy
  design.
\newblock {\em Proc. of 5th {IFAC} Conference on Analysis and Design of Hybrid
  Systems}, pages 205--210, 2015.

\bibitem{FleckNAHS2016}
J.~L. Fleck and C.~G. Cassandras.
\newblock Optimal design of personalized prostate cancer therapy using
  infinitesimal perturbation analysis.
\newblock {\em Nonlinear Analysis: Hybrid Systems}, under review.

\bibitem{Hirata2010}
Y.~Hirata, N.~Bruchovsky, and K.~Aihara.
\newblock Development of a mathematical model that predicts the outcome of
  hormone therapy for prostate cancer.
\newblock {\em J. Theor. Biology}, 264:\penalty0 517--527, 2010.

\bibitem{Hirata2010B}
Y.~Hirata, M.~di~Bernardo, N.~Bruchovsky, and K.~Aihara.
\newblock Hybrid optimal scheduling for intermittent androgen suppression of
  prostate cancer.
\newblock {\em Chaos}, 20:\penalty0 045125, 2010.

\bibitem{Ideta2008}
A.M. Ideta, G.~Tanaka, T.~Takeuchi, and K.~Aihara.
\newblock A mathematical model for intermittent androgen suppression for
  prostate cancer.
\newblock {\em J. Nonlinear Sci.}, 18:\penalty0 593--614, 2008.

\bibitem{Jackson2004}
T.L. Jackson.
\newblock A mathematical investigation of the multiple pathways to recurrent
  prostate cancer: comparison with experimental data.
\newblock {\em Neoplasia}, 6:\penalty0 697--704, 2004.

\bibitem{Liu2015}
B.~Liu, S.~Kong, S.~Gao, P.~Zuliani, and E.M. Clarke.
\newblock Towards personalized cancer therapy using delta-reachability
  analysis.
\newblock {\em HSCC2015}, 2015.

\bibitem{Harrison2012}
D.L. Longo, A.S. Fauci, D.L. Kasper, S.L. Hauser, J.L. Jameson, and
  J.~Loscalzo, editors.
\newblock {\em Harrison's principles of internal medicine}.
\newblock McGraw-Hill, Medical Pub. Division, New York, 18th edition, 2012.

\bibitem{Shimada2008}
T.~Shimada and K.~Aihara.
\newblock A nonlinear model with competition between prostate tumor cells and
  its application to intermittent androgen suppression therapy of prostate
  cancer.
\newblock {\em Mathematical Biosciences}, 214:\penalty0 134--139, 2008.

\bibitem{Suzuki2010}
T.~Suzuki, N.~Bruchovsky, and K.~Aihara.
\newblock Piecewise affine systems modelling for optimizing therapy of prostate
  cancer.
\newblock {\em Philos. Trans. R. Soc.}, 368:\penalty0 5045--5059, 2010.

\bibitem{Tanaka2010}
G.~Tanaka, Y.~Hirata, S.L. Goldenberg, N.~Bruchovsky, and K.~Aihara.
\newblock Mathematical modelling of prostate cancer growth and its application
  to hormone therapy.
\newblock {\em Philos. Trans. R. Soc.}, 368:\penalty0 5029--5044, 2010.

\bibitem{Tao2010}
Y.~Tao, Q.~Guo, and K.~Aihara.
\newblock A mathematical model of prostate tumor growth under hormone therapy
  with mutation inhibitor.
\newblock {\em J. Nonlinear Sci.}, 20:\penalty0 219--240, 2010.

\end{thebibliography}
\end{document}